\documentclass[reqno]{amsart}

\usepackage{amsmath,amssymb,amsthm,amsfonts}
\usepackage{hyperref}
\usepackage{appendix}
\usepackage{graphicx}
\usepackage{setspace}

\usepackage{color}

\newtheorem{lemma}{Lemma}[section]
\newtheorem{theorem}{Theorem}[section]

\newtheorem{remark}{Remark}[section]
\newtheorem{corollary}{Corollary}[section]

\numberwithin{equation}{section}

\begin{document}
\title[Boundary layer for plane parallel MHD flow]{Stability of the boundary layer expansion for the 3D plane parallel MHD flow }
\thanks{$^*$Corresponding author}
\thanks{{\it Keywords}: Boundary layer; plane parallel channel flows; MHD equations; Prandtl theory.}
\thanks{{\it AMS Subject Classification}: 76N10, 35Q30, 35R35}%
\author[Shijin Ding, Zhilin Lin, Dongjuan Niu]{Shijin Ding, Zhilin Lin, Dongjuan Niu$^*$}
\address[S. Ding]{South China Research Center for Applied Mathematics and Interdisciplinary Studies, South China Normal University,
Guangzhou, 510631, China}\address{School of Mathematical Sciences, South China Normal University,
Guangzhou, 510631, China}
\email{dingsj@scnu.edu.cn}
\address[Z. Lin]{School of Mathematical Sciences, South China Normal University,
Guangzhou, 510631, China}
\email{zllin@m.scnu.edu.cn}

\address[Corresponding author: D. Niu]{School of Mathematical Sciences, Capital Normal University,
Beijing, 100048, China}
\email{djniu@cnu.edu.cn}
\date{\today}

\begin{abstract}
In this paper, we establish the mathematical validity of the Prandtl boundary layer theory for a class of nonlinear plane parallel flows of viscous
incompressible magnetohydrodynamic (MHD) flow with no-slip boundary condition of velocity and perfectly conducting wall for magnetic fields. The convergence is shown under various Sobolev norms, including the physically important space-time
uniform norm $L^\infty(H^1)$. In addition, the similar convergence results are also obtained under the case with uniform magnetic fields. This implies the stabilizing effects of magnetic fields. Besides, the higher-order expansion is also considered.
\end{abstract}

\maketitle

\vspace{-5mm}

\section{Introduction}

The system of magnetohydrodynamics (MHD) is a very fundamental model to describe the motion of fluid with electromagnetic field. In this paper, we consider the viscous, incompressible MHD equations in a periodic channel $(x,y,z) \in \Omega:=\mathbb{T}^2\times [0,1](\mathbb{T}=[0,L])$
\begin{equation}\label{1.1}
\left \{
\begin{array}{lll}
\partial_t u^\varepsilon-\nu \Delta u^\varepsilon+(u^\varepsilon \cdot \nabla)u^\varepsilon-(H^\varepsilon \cdot \nabla)H^\varepsilon+\nabla p^\varepsilon=f,\\
\partial_t H^\varepsilon-\kappa\Delta H^\varepsilon+(u^\varepsilon \cdot \nabla)H^\varepsilon-(H^\varepsilon \cdot \nabla)u^\varepsilon=0,\\
\nabla \cdot u^\varepsilon=\nabla \cdot H^\varepsilon=0,
\end{array}
\right.
\end{equation}
where $u^\varepsilon, H^\varepsilon, f \in\mathbb{R}^3$ and $p^\varepsilon\in\mathbb{R}$ are the velocity fields, magnetic fields, external force and pressure, respectively. And $\nu,\kappa$ are the viscosity and resistivity coefficients, respectively. In this paper, we suppose that the viscosity and resistivity coefficients are both $\nu=\kappa=:\varepsilon >0$ for simplicity.

In this paper, we intend to describe the inviscid and vanishing resistivity limit of the MHD system in a channel, therefore it is obvious that the boundary conditions have played an important role in our problem. Motivated by the physical literatures \cite{Davidson,G-P,Gilbert,CLiu2,CLiu3} and so on, we naturally consider the following boundary conditions and initial data for \eqref{1.1}:
\begin{equation}\label{1.1b}
\left \{
\begin{array}{lll}
u^\varepsilon|_{z=i}=\alpha^i(t;x), \ \ i=0,1,\\
(\partial_z H^\varepsilon_j,H^\varepsilon_3)|_{z=i}=(0,0), \ i=0,1, \ j=1,2,\\
(u^\varepsilon,H^\varepsilon)|_{t=0}=(u_0,H_0),
\end{array}
\right.
\end{equation}
i.e., the Dirichlet boundary condition is imposed on the velocity fields and the perfectly conducting boundary condition is imposed on the magnetic fields, respectively. Here $\alpha^i=(\alpha^i_1,\alpha^i_2,\alpha^i_3)(t;x)$.
It should be pointed out that  \eqref{1.1b} is reduced to the classical no-slip boundary condition $u^\epsilon|_{\partial\Omega}=0$ provided that $\alpha^i=0$. For more details about \eqref{1.1b}, refer to \cite{Davidson,G-P,Gilbert,CLiu2,CLiu3} and the references therein for instance.

Formally letting $\varepsilon \to 0$, we obtain the ideal MHD equations
\begin{equation}\label{1.1c}
\left \{
\begin{array}{lll}
\partial_t u^0+(u^0 \cdot \nabla)u^0-(H^0 \cdot \nabla)H^0+\nabla p^0=f,\\
\partial_t H^0+(u^0 \cdot \nabla)H^0-(H^0 \cdot \nabla)u^0=0,\\
\nabla \cdot u^0=\nabla \cdot H^0=0.
\end{array}
\right.
\end{equation}
For \eqref{1.1c}, the boundary conditions and initial data \eqref{1.1b} was correspondingly supplied with the following
\begin{equation}\label{1.1d}
\left \{
\begin{array}{lll}
u^0_3|_{z=i}=0, \ \ i=0,1,\\
H^0_3|_{z=i}=0, \ i=0,1, \\
(u^0,H^0)|_{t=0}=(u_0,H_0).
\end{array}
\right.
\end{equation}

Similar to inviscid limit process from the Navier-Stokes equations to Euler equations, there is a mismatch between the tangential components of the velocity and magnetic fields, which is  so-called  the boundary layer introduced by Prandtl \cite{Prandtl} in 1900s.  More precisely, there is a thin layer of width of order $\sqrt{\varepsilon}$ near the boundary, where the viscous solutions $(u^\varepsilon,H^\varepsilon)$ change dramatically from the boundary data to the ideal MHD flow $(u^0,H^0)$.  The solution of Prandtl is that the ``viscous" solution can be decomposed into three parts: the ``inviscid" solution, boundary layer solutions and the higher order remainder terms. Following the idea of Prandtl, we can similarly deal with the MHD flow, see \cite{CLiu2,CLiu3} and the references therein for more discussions. In a word,  the solution of (\ref{1.1}) can be decomposed into the following form:
\begin{equation}\label{1.7}
\mathrm{Viscous \ \ MHD}\backsimeq\mathrm{Ideal \ \ MHD}+\mathrm{Boundary \ \ layer}+O(\sqrt{\epsilon}).
\end{equation}
To verify the above expansion of (\ref{1.7}), there are at least the following problems to be investigated:

(a) The well-posedness of the boundary layer equations;

(b) The justification of (\ref{1.7}) (or the validity of the boundary layer expansion), including the convergence for the error solutions, which is the difference between viscous solutions and the approximations.

Due to the strong coupling between the magnetic fields and velocity fields in MHD system,
there are few results about this topic compared with those of Navier-Stokes equations. For the steady MHD without magnetic diffusion, Wang and Ma \cite{JWang} showed that well-posedness for the steady MHD boundary layer equations.  A great progress is obtained by Liu, Xie and Yang \cite{CLiu2,CLiu3}. Without the classical Oleinik monotone condition, the local in time well-posedness and validity for the two-dimensional MHD boundary layer problem are obtained under the condition that the tangential magnetic field is non-degenerate. The vanishing viscosity limit of the 3D viscous MHD system in a class of bounded domains with slip boundary conditions was obtained by Xiao, Xin and Wu \cite{YXiao}. With the standard Dirichlet boundary conditions, Wang and Xin \cite{SWang} established the uniform stability of the Prandtl's type boundary layers for a special non-trivial class of initial data. The viscosity vanishing limit for the nonlinear pipe magnetohydrodynamic flow with fixed magnetic diffusion was shown by Wu and Wang \cite{ZWu}. The boundary layer problem for the incompressible
MHD system with non-characteristic perfect
conducting wall boundary condition was studied by Wang and Wang \cite{SWang0}.
For MHD plane parallel flow with different viscosity and magnetic diffusion, Wang and Wang \cite{NWang} studied the boundary layer theory about the nonhomegenous Dirichlet boundary conditions and proved the convergence in $L^2(H^1)$ sense. Very recently, the validity of the boundary layer theory for the steady viscous MHD flow as both viscosity and magnetic diffusion vanishing has been conducted in a work by the first two authors and Xie \cite{Ding2}.

Motivated by \cite{CLiu2,CLiu3}, we intend to investigate the MHD system with the same boundary conditions as \cite{CLiu3} in 3D channel, i.e., \eqref{1.1}-\eqref{1.1c}.  Compared with the weak boundary layer problem studied in \cite{YXiao}, it is very interesting to study the Prandtl layer theory with no-slip boundary conditions, which will generate the strong boundary layer. Additionally, another problem is that whether the Prandtl layer theory can be global in time or not. In this work, we intend to study a special case, i.e., the plane parallel flow in the MHD equations, which can be found in \cite{Mazzucato,XWang}. Precisely, we consider the boundary layer for the MHD plane parallel flow of the form
\begin{equation}\label{1.2}
u^\varepsilon=(u^\varepsilon_1(t;z),u^\varepsilon_2(t;x,z),0), \ H^\varepsilon=(H^\varepsilon_1(t;z),H^\varepsilon_2(t;x,z),0), \ p^\varepsilon=\mathrm{constant}.
\end{equation}
It is easy to check that (\ref{1.2}) is the solution of (\ref{1.1}) with the external force
$$f=(f_1(t;z),f_2(t;x,z),0),$$
which satisfy that
\begin{equation}\label{1.4}
\left \{
\begin{array}{lll}
\partial_t u^\varepsilon_1-\varepsilon \partial_{zz} u^\varepsilon_1=f_1,\\
\partial_t u^\varepsilon_2-\varepsilon \Delta_{x,z}u^\varepsilon_2 +u^\varepsilon_1\partial_x u^\varepsilon_2-H^\varepsilon_1\partial_x H^\varepsilon_2=f_2,\\
\partial_t H^\varepsilon_1-\varepsilon \partial_{zz} H^\varepsilon_1=0,\\
\partial_t H^\varepsilon_2-\varepsilon \Delta_{x,z}H^\varepsilon_2 +u^\varepsilon_1\partial_x H^\varepsilon_2-H^\varepsilon_1\partial_x u^\varepsilon_2=0,\\
\end{array}
\right.
\end{equation}
for $(x,z)\in [0,L]\times [0,1]$, here $\Delta_{x,z}=\partial_{xx}+\partial_{zz}$.
Under the ansatz of the plane parallel MHD flow, the boundary conditions and initial data are imposed as follows:
\begin{equation}\label{1.3}
\left \{
\begin{array}{lll}
u^\varepsilon|_{z=i}=\alpha^i(t;x), \ \ \alpha^i(t;x)=(\alpha^i_1(t),\alpha^i_2(t;x),0), \ \ i=0,1,\\
(\partial_z H^\varepsilon_j,H^\varepsilon_3)|_{z=i}=(0,0), \ i=0,1, \ j=1,2,\\
(u^\varepsilon,H^\varepsilon)|_{t=0}=(u_0,H_0),\ u_0=(a(z),b(x,z),0),\ H_0=(c(z),d(x,z),0).
\end{array}
\right.
\end{equation}

It should be emphasized that the plane parallel flows that we consider here are three-dimensional actually. Thanks to the weak coupling of system (\ref{1.4})-\eqref{1.3}, the well-posedness can be obtained easily. For example, we can get that $u^\varepsilon \in L^\infty(0,T;H^1(\Omega))$ and $H^\varepsilon \in L^\infty(0,T;H^1(\Omega))$ provided that $(u_0,H_0)\in H^1(\Omega) \times H^1(\Omega)$, $\alpha^i,\gamma^i\in H^1$ and $f \in L^\infty(0,T;H^1(\Omega))$. Here we skip more details and refer to \cite{Michel} for the interested readers.

Let $\varepsilon \to 0$, one gets the ideal MHD equations
\begin{equation}\label{1.5}
\left \{
\begin{array}{lll}
\partial_t u^0_1=f_1,\\
\partial_t u^0_2 +u^0_1\partial_x u^0_2-H^0_1\partial_x H^0_2=f_2,\\
\partial_t H^0_1=0,\\
\partial_t H^0_2 +u^0_1\partial_x H^0_2-H^0_1\partial_x u^0_2=0,\\
\end{array}
\right.
\end{equation}
with the following boundary conditions and the same initial data
\begin{equation}\label{1.6}
\left \{
\begin{array}{lll}
(u^0_3,H^0_3)|_{z=i}=(0,0), \ i=0,1,\\
(u^0,H^0)|_{t=0}=(u_0,H_0).
\end{array}
\right.
\end{equation}
We deduce from (\ref{1.5}) and (\ref{1.6}) that
$$H^0_1(t;z)=H^0_1(z)\equiv H_{0,1}(z).$$
Moreover, the $(u^0_2,H^0_2)$ can be determined by a linear hyperbolic system with the initial data. As an example, one can deduce that $u^0 \in L^\infty(0,T;H^1(\Omega))$ and $H^0 \in L^\infty(0,T;H^1(\Omega))$ provided that $(u_0,H_0)\in H^1(\Omega) \times H^1(\Omega)$ and $f \in L^\infty(0,T;H^1(\Omega))$.
More about the well-posedness and regularity theory of ideal MHD can be found in \cite{Secchi,Yanagisawa} and the references therein.

It is necessary to discuss the challenges in our problem. First, due to the appearance of the magnetic field, one should construct the approximate solutions not only for the velocity field, but also for the magnetic field. Thanks to the good structure of the flow, which leads to the weak coupling between the velocity and magnetic fields, we are possible to decouple the two boundary layers in order. Furthermore, it should be pointed out that the constructions of the approximate solutions for magnetic fields are different from that for velocity fields due to the perfectly conducting boundary condition. Precisely, the perfectly conducting conditions would result in an addition boundary term of the ideal MHD flow, which is not the same order of $\varepsilon$ as that of the boundary layer correctors. To overcome the difficulties, the key idea is to introduce the boundary corrector, which cancels the boundary terms resulted from the ideal MHD flows. This idea may be applied in more other problems such as \cite{Ding2,CLiu3}. More details about the constructions of the approximate solutions can be found in Section \ref{approximate}.
Fortunately, we prove the convergence rate in $L^\infty$ norm in Theorem \ref{thm1} is the order of $\sqrt{\varepsilon}$ and all the results are global in time. In addition, we also obtain the higher convergence rate in $L^\infty(H^1)$ norm sense. This may be resulted from stabilizing effect of the magnetic fields and the good structure of the plane parallel MHD flow.

Besides the perfectly conducting wall condition of the magnetic fields in \eqref{1.1c}, we are also interested in the no-slip boundary conditions to \eqref{1.1}, i.e.,
\begin{equation}\label{1.3a:nonslip}
\left \{
\begin{array}{lll}
u^\varepsilon|_{z=i}=\alpha^i(t;x), \ \ \alpha^i(t;x)=(\alpha^i_1(t),\alpha^i_2(t;x),0), \ \ i=0,1,\\
H^\varepsilon|_{z=i}=\gamma^i(t;x), \ \ \gamma^i(t;x)=(\gamma^i_1(t),\gamma^i_2(t;x),0), \ \ i=0,1,\\
(u^\varepsilon,H^\varepsilon)|_{t=0}=(u_0,H_0),\ u_0=(a(z),b(x,z),0),\ H_0=(c(z),d(x,z),0),
\end{array}
\right.
\end{equation}
which will be described in detail in Section \ref{stabilityeffect}.

It is our second concern to verify the convergence of solution of \eqref{1.1} with uniform magnetic background to that of \eqref{1.1c} when $\varepsilon$ vanishes, which is equivalent to vanishing viscosity and resistivity limit of \eqref{1.1} with nonhomogeneous Dirichlet condition \eqref{1.3a:nonslip}. In this case, the boundary condition for the magnetic fields can be considered as no-slip boundary condition, which may generate a strong boundary layer effect. To our knowledge, for general case, even in 2-D problem, it still remains open because it is challenging to control the behaviour of the vorticity (or the stream function) of the magnetic fields near the boundaries. In this paper, we attempt to answer this type of problem and try to find out the stabilizing effects of the magnetic fields, although by means of the structure of plane MHD flow and the uniformly magnetic fields. It is mentioned that the construction for the approximate solutions here may be more simple than that case with perfectly conducting wall, since there is no any additional boundary terms resulted from the ideal MHD flows and we only need to match the order of $\varepsilon$ for the boundary conditions. Compared with the previous work \cite{G-P,CLiu0}, which developed this problem in the linear level with the perturbation around some shear flows, we are inspiring in Theorem \ref{thm:nonslip} to prove the same convergence rate as the case of the perfectly conducting wall conditions, which includes the higher Sobolev norm $L^\infty(H^1)$ and gives us a clue about the stabilizing effects of magnetic fields with the structure of plane parallel flow. In addition, we also mention that it is very difficult to study the problem in general case, even for the 2-D well-posedness of MHD boundary layer equations because we lose to control the tangential derivatives and the additional boundary effect resulted from the magnetic fields.

Before we state the main theorem, we mention that there are some compatibility conditions imposed to ensure the higher order regularity of \eqref{1.4}-\eqref{1.3}. More precisely, the zero-order compatibility conditions would be given as follows:
\begin{equation}\label{1.4a}
\left \{
\begin{array}{lll}
\alpha^i(0;x)=u_0(x,i), \ \ i=0,1,\\
(H_0)_3(x,0)=0,\\
\partial_z (H_0)_j(x,i)=0,\ \ i=0,1, \ \ j=1,2.
\end{array}
\right.
\end{equation}
In addition, the first-order compatibility conditions would be given as follows:
\begin{equation}\label{1.4b}
\left \{
\begin{array}{lll}
\partial_t \alpha^i_1(0)-\varepsilon \partial_{zz}a(i)=f_1(0;i),\\
\partial_t \alpha^i_2(0;x)-\varepsilon \Delta_{x,z}b(x,i)+a(i)\partial_x b(x,i)-c(i)\partial_x d(x,i)=f_2(0;x,i),\\
-\varepsilon \partial_{zzz}c(i)=0,\\
-\varepsilon \Delta_{x,z}\partial_z d(x,i)+\partial_z \left[a(i)\partial_x d(x,i)-c(i)\partial_x b(x,i)\right]=0,
\end{array}
\right.
\end{equation}
where $i=0,1$.

Our main result is stated as follows.
\begin{theorem}\label{thm1}
Suppose that $u_0, H_0 \in H^m(\Omega)$, $f \in L^\infty(0,T;H^m(\Omega))$, \linebreak $\alpha^i \in H^2(0,T; H^m(\partial{\Omega}))$, $i=0,1, m>5$, satisfying the compatibility condition (\ref{1.4a}). Then there exist positive constants $C>0$, independent of $\varepsilon$, such that for any solution $(u^\varepsilon,H^\varepsilon)$ of (\ref{1.4}) with the initial data $(u_0,H_0)$ and boundary data $\alpha^i$ in \eqref{1.3}, satisfying that
\begin{equation}\label{1.73}
\Vert (u^\varepsilon-\tilde{u}^a,H^\varepsilon-\tilde{h}^a) \Vert_{L^\infty(0,T;L^2(\Omega))} \leq C \varepsilon^{\frac{3}{4}},
\end{equation}
\begin{equation}\label{1.74}
\Vert (u^\varepsilon-\tilde{u}^a,H^\varepsilon-\tilde{h}^a) \Vert_{L^\infty(0,T;H^1(\Omega))} \leq C \varepsilon^{\frac{1}{4}},
\end{equation}
\begin{equation}\label{1.75}
\Vert (u^\varepsilon-\tilde{u}^a,H^\varepsilon-\tilde{h}^a)\Vert_{L^\infty((0,T) \times \Omega))} \leq C \sqrt{\varepsilon},
\end{equation}
where $\tilde{u}^a,\tilde{h}^a$ are defined by (\ref{3.17}) in Section \ref{approximate}.
\end{theorem}

\begin{remark}
Our main theorem, Theorem \ref{thm1}, shows that the convergence rates obtained in Navier-Stokes equations can be extended to the MHD flow. It is noted that the convergence rates for the magnetic field is the same as for the velocity field. It seems that the boundary layer is essentially resulted from the mismatch of the boundary conditions for the velocity field.
\end{remark}

\begin{remark}
Indeed, due to the good structure of the plane MHD flow, the leading order profiles in the approximate solutions will obey weak coupling system, which inspires that the convergence theory can be established and the results are global in time.
\end{remark}
\begin{remark}
Indeed, our results can be extended to the case that the viscosity and resistivity coefficients are $\nu\varepsilon$ and $\kappa\varepsilon$ with constants $\nu,\kappa>0$, respectively, which is also studied in \cite{CLiu2,CLiu3}.
\end{remark}

Through this paper, we denote
$$\langle Z\rangle :=\sqrt{1+|Z|^2}.$$

The rest of this paper is organized as follows: Section \ref{boundarylayerequ} is devoted to deriving the boundary layer equations and boundary conditions for the correctors; the modified approximate solutions will be constructed in Section \ref{approximate}; our main theorem will be proved in Section \ref{pf}; Section \ref{stabilityeffect} is devoted to studying the problem with uniform magnetic background; higher order expansions and improved convergence rates will be obtained in Section \ref{improved}.

\section{The boundary layer equations for the correctors}\label{boundarylayerequ}
In this section, we will derive the boundary layer type equations for the correctors. The approach to a rigorous boundary layer analysis that we take is to derive the equations for the correctors, which is the difference between the viscous MHD solutions $(u^\varepsilon,H^\varepsilon,0)$ and the ideal MHD solutions $(u^0,H^0,0)$,
We assume that the viscous MHD solutions are well approximated by
\begin{equation}\label{2.1}
\left \{
\begin{array}{lll}
u^a_1(t;z)&:=u^{ou}_1(t;z)+\theta^0_1\left(t;\frac{z}{\sqrt{\varepsilon}}\right)+\theta^{u,0}_1\left(t;\frac{1-z}{\sqrt{\varepsilon}}\right),\\
u^a_2(t;x,z)&:=u^{ou}_2(t;x,z)+\theta^0_2\left(t;x,\frac{z}{\sqrt{\varepsilon}}\right)+\theta^{u,0}_2\left(t;x,\frac{1-z}{\sqrt{\varepsilon}}\right),\\
h^a_1(t;z)&:=H^{ou}_1(t;z)+h^0_1\left(t;\frac{z}{\sqrt{\varepsilon}}\right)+h^{u,0}_1\left(t;\frac{1-z}{\sqrt{\varepsilon}}\right),\\
h^a_2(t;x,z)&:=H^{ou}_2(t;x,z)+h^0_2\left(t;x,\frac{z}{\sqrt{\varepsilon}}\right)+h^{u,0}_2\left(t;x,\frac{1-z}{\sqrt{\varepsilon}}\right),
\end{array}
\right.
\end{equation}
where the correctors satisfy that
\begin{equation}\label{2.1}
(\theta^0_i,h^0_i) \to 0 \ \ \mathrm{as} \ Z \to \infty; \ \ \ (\theta^{u,0}_i,h^{u,0}_i) \to 0 \ \ \mathrm{as} \ Z^u \to \infty,
\end{equation}
in which $i=1,2, Z=z/\sqrt{\varepsilon}$ and $Z^u=(1-z)/\sqrt{\varepsilon}$.

Every part in the approximate solutions satisfy the following problems:

\textbf{(I) The outer solutions $(u^{ou},H^{ou})$.}

The outer solutions $(u^{ou},H^{ou})$ satisfy the ideal MHD equations (\ref{1.5}) with the initial data
\begin{equation}\label{2.2}
(u^0,H^0)|_{t=0}=(u_0,H_0).
\end{equation}
The uniqueness of the solutions to the system implies that $(u^{ou},H^{ou})\equiv (u^0,H^0)$.

\textbf{(II) The lower correctors $(\theta^0_1,\theta^0_2,h^0_1,h^0_2)$.}

The lower correctors $(\theta^0_1,\theta^0_2,h^0_1,h^0_2)$ satisfy that
\begin{equation}\label{2.3}
\left \{
\begin{array}{lll}
\partial_t \theta^0_1-\partial_{ZZ}\theta^0_1=0,\\
\partial_t \theta^0_2-\partial_{ZZ}\theta^0_2+(u^0_1(t;0)+\theta^0_1)\partial_x \theta^0_2+\theta^0_1\partial_x u^0_2(t;x,0)\\
\quad\quad\quad  -(H^0_1(t;0)+h^0_1)\partial_x h^0_2-h^0_1\partial_x H^0_2(t;x,0)=0,\\
\partial_t h^0_1-\partial_{ZZ}h^0_1=0,\\
\partial_t h^0_2-\partial_{ZZ}h^0_2+(u^0_1(t;0)+\theta^0_1)\partial_x h^0_2+\theta^0_1\partial_x H^0_2(t;x,0)\\
\quad\quad\quad  -(H^0_1(t;0)+h^0_1)\partial_x \theta^0_2-h^0_1\partial_x u^0_2(t;x,0)=0,\\
(\theta^0_1,\theta^0_2,\partial_Z h^0_1,\partial_Z h^0_2)|_{Z=0}=(\alpha^0_1(t)-u^0_1(t;0),\alpha^0_2(t;x)-u^0_2(t;x,0),0,0),\\
(\theta^0_1,\theta^0_2, h^0_1,h^0_2)|_{Z=\infty}=(0,0,0,0),\\
(\theta^0_1,\theta^0_2, h^0_1,h^0_2)|_{t=0}=(0,0,0,0).
\end{array}
\right.
\end{equation}

\textbf{(III) The upper correctors $(\theta^{u,0}_1,\theta^{u,0}_2,h^{u,0}_1,h^{u,0}_2)$.}

The lower correctors $(\theta^{u,0}_1,\theta^{u,0}_2,h^{u,0}_1,h^{u,0}_2)$ satisfy that
\begin{equation}\label{2.4}
\left \{
\begin{array}{lll}
\partial_t \theta^{u,0}_1-\partial_{Z^uZ^u}\theta^{u,0}_1=0,\\
\partial_t \theta^{u,0}_2-\partial_{Z^uZ^u}\theta^{u,0}_2+(u^0_1(t;1)+\theta^{u,0}_1)\partial_x \theta^{u,0}_2+\theta^{u,0}_1\partial_x u^0_2(t;x,1)\\
\quad\quad\quad  -(H^0_1(t;1)+h^{u,0}_1)\partial_x h^{u,0}_2-h^{u,0}_1\partial_x h^{u,0}_2(t;x,1)=0,\\
\partial_t h^{u,0}_1-\partial_{Z^uZ^u}h^{u,0}_1=0,\\
\partial_t h^{u,0}_2-\partial_{Z^uZ^u}h^{u,0}_2+(u^0_1(t;1)+\theta^{u,0}_1)\partial_x h^{u,0}_2+\theta^{u,0}_1\partial_x h^{u,0}_2(t;x,1)\\
\quad\quad\quad  -(H^0_1(t;1)+h^{u,0}_1)\partial_x \theta^{u,0}_2-h^{u,0}_1\partial_x u^0_2(t;x,1)=0,\\
(\theta^{u,0}_1,\theta^{u,0}_2,\partial_{Z^u} h^{u,0}_1,\partial_{Z^u} h^{u,0}_2)|_{Z=0}\\
 \qquad\qquad \qquad=(\alpha^1_1(t)-u^1_1(t;1),\alpha^1_2(t;x)-u^0_2(t;x,1),0,0),\\
(\theta^{u,0}_1,\theta^{u,0}_2, h^{u,0}_1,h^{u,0}_2)|_{Z^u=\infty}=(0,0,0,0),\\
(\theta^{u,0}_1,\theta^{u,0}_2, h^{u,0}_1,h^{u,0}_2)|_{t=0}=(0,0,0,0).
\end{array}
\right.
\end{equation}

Due to the symmetry of the lower and upper correctors, for simplicity, we would only discuss the problem (\ref{2.3}). As shown in (\ref{2.3}), $\theta^0_1, h^0_1$ satisfy the one-dimensional heat equation with Dirichlet boundary conditions and Neumann boundary condition on $\{Z=0\}$, respectively. Therefore the well-posedness and regularity results are classical (\cite{Evans}). For the problems for $\theta^0_2,h^0_2$, they satisfy a parabolic system, which can be solved by modifying the methods in \cite{Xin1}. Then the full problem (\ref{2.3}) is well-posed, so as for (\ref{2.4}). In addition, all the weighted estimates used in our arguments can be obtained by applying standard energy arguments, here we refer to \cite{Mazzucato} for details and we omit the proof here.

\section{The approximate solutions}\label{approximate}
In this section, we will construct the approximate solutions.
Based on the arguments in Section \ref{boundarylayerequ}, we can know that the each corrector is well-defined, then the approximations are well-defined. To study our problem, we introduce the modified approximate solutions with a cut-off function, which can be found in \cite{Mazzucato} for instance.

Let $\psi(z)$ be a smooth function on $[0,1]$ with
\begin{equation}\label{3.1}
\psi(z)=\left \{
\begin{array}{lll}
1,& z\in [0,\frac{1}{3}],\\
0,& z \in [\frac{1}{2},1],\\
\mathrm{smooth},& \mathrm{otherwise}.
\end{array}
\right.
\end{equation}
It is easy to check that $\psi(z)\psi(1-z)=0$ for any $z\in [0,1]$.

We introduce the truncated approximations as follows
\begin{equation}\label{3.2}
\left \{
\begin{array}{lll}
\tilde{u}^a_1(t;z)&:=u^{0}_1(t;z)+\psi(z)\theta^0_1\left(t;\frac{z}{\sqrt{\varepsilon}}\right)+\psi(1-z)\theta^{u,0}_1\left(t;\frac{1-z}{\sqrt{\varepsilon}}\right),\\
\tilde{u}^a_2(t;x,z)&:=u^{0}_2(t;x,z)+\psi(z)\theta^0_2\left(t;x,\frac{z}{\sqrt{\varepsilon}}\right)+\psi(1-z)\theta^{u,0}_2\left(t;x,\frac{1-z}{\sqrt{\varepsilon}}\right),\\
\tilde{h}^a_1(t;z)&:=H^{0}_1(t;z)+\psi(z)h^0_1\left(t;\frac{z}{\sqrt{\varepsilon}}\right)+\psi(1-z)h^{u,0}_1\left(t;\frac{1-z}{\sqrt{\varepsilon}}\right),\\
\tilde{h}^a_2(t;x,z)&:=H^{0}_2(t;x,z)+\psi(z)h^0_2\left(t;x,\frac{z}{\sqrt{\varepsilon}}\right)+\psi(1-z)h^{u,0}_2\left(t;x,\frac{1-z}{\sqrt{\varepsilon}}\right),
\end{array}
\right.
\end{equation}
then the $(\tilde{u}^a,\tilde{H}^a)$ satisfy that
\begin{equation}\label{3.3}
\left \{
\begin{array}{lll}
\partial_t \tilde{u}^a_1-\varepsilon \partial_{zz} \tilde{u}^a_1=f_1+A+B,\\
\partial_t \tilde{u}^a_2-\varepsilon \Delta_{x,z}\tilde{u}^a_2 +\tilde{u}^a_1\partial_x \tilde{u}^a_2-\tilde{h}^a_1\partial_x \tilde{h}^a_2=f_2+C+D+E,\\
\partial_t \tilde{h}^a_1-\varepsilon \partial_{zz} \tilde{h}^a_1=F+G,\\
\partial_t \tilde{h}^a_2-\varepsilon \Delta_{x,z}\tilde{h}^a_2 +\tilde{u}^a_1\partial_x \tilde{h}^a_2-\tilde{h}^a_1\partial_x \tilde{u}^a_2=H+I+J,\\
\end{array}
\right.
\end{equation}
where the remainders are given by
\begin{equation}\label{3.5}
A=-2\sqrt{\varepsilon}\left(\psi'(z)\partial_Z \theta^0_1+\psi'(1-z)\partial_{Z^u}\theta^{u,0}_1\right),
\end{equation}
\begin{equation}\label{3.6}
B=-\varepsilon\left(\partial_{zz}u^0_1+\psi''(z)\theta^0_1+\psi''(1-z)\theta^{u,0}_1\right),
\end{equation}
\begin{equation}\label{3.7}
\begin{array}{lll}
C=&\psi(z)(\psi(z)-1)\left(\theta^0_1\partial_x\theta^0_2-h^0_1\partial_xh^0_2\right)\\
&+\psi(1-z)(\psi(1-z)-1)\left(\theta^{u,0}_1\partial_x\theta^{u,0}_2-h^{u,0}_1\partial_xh^{u,0}_2\right),
\end{array}
\end{equation}
\begin{equation}\label{3.8}
\begin{array}{lll}
D=&\sqrt{\varepsilon}\Big[\psi(z)\big(Z\partial_z u^0_1(t;0)\partial_x \theta^0_2+Z\theta^0_1\partial_{zx}u^0_2(t;x,0)-Z\partial_z H^0_1(t;0)\partial_x h^0_2\\
&-Z\partial_{zx}H^0_2(t;x,0)h^0_1\big)-\psi(1-z)\big(Z^u\partial_z u^0_1(t;1)\partial_x \theta^{u,0}_2\\
&+Z^u\theta^{u,0}_1\partial_{zx}u^0_2(t;x,1)-Z^u\partial_z H^0_1(t;1)\partial_x h^{u,0}_2-Z^u\partial_{zx}H^0_2(t;x,1)h^{u,0}_1\big)\\
&-2\psi'(z)\partial_Z\theta^0_2-2\psi'(1-z)\partial_{Z^u}\theta^{u,0}_2\Big],
\end{array}
\end{equation}
\begin{equation}\label{3.9}
\begin{array}{lll}
E=&\varepsilon\bigg(-\Delta_{x,z}u^0_2-\psi(z)\partial_{xx}\theta^0_2\\
&-\psi(1-z)\partial_{xx}\theta^{u,0}_2-\psi''(z)\theta^0_2-\psi''(1-z)\theta^{u,0}_2\bigg),
\end{array}
\end{equation}
\begin{equation}\label{3.10}
F=-2\sqrt{\varepsilon}\left(\psi'(z)\partial_Z h^0_1+\psi'(1-z)\partial_{Z^u}h^{u,0}_1\right),
\end{equation}
\begin{equation}\label{3.11}
G=-\varepsilon\left(\partial_{zz}H^0_1+\psi''(z)h^0_1+\psi''(1-z)h^{u,0}_1\right),
\end{equation}
\begin{equation}\label{3.12}
\begin{array}{lll}
H=&\psi(z)(\psi(z)-1)\left(\theta^0_1\partial_x h^0_2-h^0_1\partial_x\theta^0_2\right)\\
&+\psi(1-z)(\psi(1-z)-1)\left(\theta^{u,0}_1\partial_xh^{u,0}_2-h^{u,0}_1\partial_xh^{u,0}_2\right),
\end{array}
\end{equation}
\begin{equation}\label{3.13}
\begin{array}{lll}
I=&\sqrt{\varepsilon}\Big[\psi(z)\big(Z\partial_z u^0_1(t;0)\partial_x h^0_2+Z\theta^0_1\partial_{zx}H^0_2(t;x,0)-Z\partial_z H^0_1(t;0)\partial_x \theta^0_2\\
&-Z\partial_{zx}u^0_2(t;x,0)h^0_1\big)-\psi(1-z)\big(Z^u\partial_z u^0_1(t;1)\partial_x h^{u,0}_2\\
&+Z^u\theta^{u,0}_1\partial_{zx}H^0_2(t;x,1)-Z^u\partial_z H^0_1(t;1)\partial_x \theta^{u,0}_2-Z^u\partial_{zx}u^0_2(t;x,1)h^{u,0}_1\big)\\
&-2\psi'(z)\partial_Zh^0_2-2\psi'(1-z)\partial_{Z^u}h^{u,0}_2\Big],
\end{array}
\end{equation}
\begin{equation}\label{3.14}
\begin{array}{lll}
J=&\varepsilon\bigg(-\Delta_{x,z}H^0_2-\psi(z)\partial_{xx}h^0_2\\
&-\psi(1-z)\partial_{xx}h^{u,0}_2-\psi''(z)h^0_2-\psi''(1-z)h^{u,0}_2\bigg).
\end{array}
\end{equation}

After the above discussions, the initial data and the boundary conditions for approximate solutions are
\begin{equation}\label{3.4}
\left \{
\begin{array}{lll}
(\tilde{u}^a,\tilde{h}^a)|_{t=0}=(u_0,H_0),\\
\tilde{u}^a|_{z=i}=\alpha^i(t;x), \ \ i=0,1.
\end{array}
\right.
\end{equation}
However, the perfectly conducting wall condition for the magnetic fields can not be preserved in the above constructions, i.e., $\partial_z \tilde{h}^a_j|_{z=i}\not=0, i=0,1, j=1,2$.

Due to boundary condition for the magnetic fields is the perfectly conducting condition, we shall introduce a boundary corrector to ensure that the perfectly conducting conditions can be preserved in our approximation.
Introduce a smooth cut-off function $\rho^0(Z)$ in $Z\in [0,\infty)$ with
\begin{equation}\label{3.15}
\rho^0(Z)=\left \{
\begin{array}{lll}
1,& Z\in [0,1],\\
0,& Z \in [2,\infty),\\
\mathrm{smooth},& \mathrm{otherwise},
\end{array}
\right.
\end{equation}
where $Z=\frac{z}{\sqrt{\varepsilon}}.$
Define a boundary corrector $\eta^0=(\eta^0_1,\eta^0_2)$ as
$$\eta^0_1:=-\partial_z H^0_1(0)Z\rho^0(Z), \ \eta^0_2:=-\partial_z H^0_2(t;x,0)Z\rho^0(Z),$$
then we have
$$\partial_Z \eta^0|_{Z=0}=-(\partial_z H^0_1(0),\partial_z H^0_2(t;x,0)).$$
Similar arguments can be applied in $Z^u\in [0,\infty)$: let
$\rho^{u,0}(Z^u)$ in $Z^u\in [0,\infty)$ with
\begin{equation}\label{3.16}
\rho^{u,0}(Z^u)=\left \{
\begin{array}{lll}
1,& Z^u\in [0,1],\\
0,& Z^u \in [2,\infty),\\
\mathrm{smooth},& \mathrm{otherwise},
\end{array}
\right.
\end{equation}
where $Z^u=\frac{1-z}{\sqrt{\varepsilon}}.$
The boundary corrector $\eta^{u,0}=(\eta^{u,0}_1,\eta^{u,0}_2)$ can be defined as
$$\eta^{u,0}_1:=-\partial_z H^0_1(1)Z^u\rho^{u,0}(Z^u), \ \eta^{u,0}_2:=-\partial_z H^0_2(t;x,1)Z^u\rho^{u,0}(Z^u),$$
It is easy to see that $\eta^0,\eta^{u,0}$ are regular enough due to our assumption on the ideal MHD flows $(u^0,H^0)$ and the properties of the cut-off functions.

With this, define
$$\widetilde{h^0_1}(t;Z):=h^0_1(t;Z)+\sqrt{\varepsilon}\eta^0_1(Z), \ \ \widetilde{h^0_2}(t;x,Z):=h^0_2(t;x,Z)+\sqrt{\varepsilon}\eta^0_2(t;x,Z),$$
$$\widetilde{h^{u,0}_1}(t;Z^u):=h^{u,0}_1(t;Z^u)+\sqrt{\varepsilon}\eta^{u,0}_1(Z^u), \ \ \widetilde{h^{u,0}_2}(t;x,Z^u):=h^{u,0}_2(t;x,Z^u)+\sqrt{\varepsilon}\eta^{u,0}_2(t;x,Z^u),$$
and the approximate solutions can be constructed as
\begin{equation}\label{3.17}
\left \{
\begin{array}{lll}
\tilde{u}^a_1(t;z)&:=u^{0}_1(t;z)+\psi(z)\theta^0_1\left(t;\frac{z}{\sqrt{\varepsilon}}\right)+\psi(1-z)\theta^{u,0}_1\left(t;\frac{1-z}{\sqrt{\varepsilon}}\right),\\
\tilde{u}^a_2(t;x,z)&:=u^{0}_2(t;x,z)+\psi(z)\theta^0_2\left(t;x,\frac{z}{\sqrt{\varepsilon}}\right)+\psi(1-z)\theta^{u,0}_2\left(t;x,\frac{1-z}{\sqrt{\varepsilon}}\right),\\
\tilde{h}^a_1(t;z)&:=H^{0}_1(t;z)+\psi(z)\widetilde{h^0_1}\left(t;\frac{z}{\sqrt{\varepsilon}}\right)+\psi(1-z)\widetilde{h^{u,0}_1}\left(t;\frac{1-z}{\sqrt{\varepsilon}}\right),\\
\tilde{h}^a_2(t;x,z)&:=H^{0}_2(t;x,z)+\psi(z)\widetilde{h^0_2}\left(t;x,\frac{z}{\sqrt{\varepsilon}}\right)
+\psi(1-z)\widetilde{h^{u,0}_2}\left(t;x,\frac{1-z}{\sqrt{\varepsilon}}\right),
\end{array}
\right.
\end{equation}
where $\psi$ is defined as \eqref{3.1} and we still use the $(\tilde{u}^a,\tilde{h}^a)$ to define the new approximate solutions for simplicity.
Then we have
$$(\partial_z \tilde{h}^a_1,\partial_z\tilde{h}^a_2)|_{z=i}=(0,0), \ \ i=0,1.$$
Therefore one can derive the equations for the approximate solutions as
\begin{equation}\label{3.18}
\left \{
\begin{array}{lll}
\partial_t \tilde{u}^a_1-\varepsilon \partial_{zz} \tilde{u}^a_1=f_1+A+B,\\
\partial_t \tilde{u}^a_2-\varepsilon \Delta_{x,z}\tilde{u}^a_2 +\tilde{u}^a_1\partial_x \tilde{u}^a_2-\tilde{h}^a_1\partial_x \tilde{h}^a_2=f_2+C+D_1+E_1,\\
\partial_t \tilde{h}^a_1-\varepsilon \partial_{zz} \tilde{h}^a_1=F_1+G_1+G_2,\\
\partial_t \tilde{h}^a_2-\varepsilon \Delta_{x,z}\tilde{h}^a_2 +\tilde{u}^a_1\partial_x \tilde{h}^a_2-\tilde{h}^a_1\partial_x \tilde{u}^a_2=H+I_1+J_1+J_2,\\
\end{array}
\right.
\end{equation}
with the following initial and boundary conditions
\begin{equation}\label{3.19}
\left \{
\begin{array}{lll}
(\tilde{u}^a,\tilde{h}^a)|_{t=0}=(u_0,H_0),\\
\tilde{u}^a|_{z=i}=\alpha^i(t;x), \ \ i=0,1,\\
(\partial_z \tilde{h}^a_1,\partial_z\tilde{h}^a_2)|_{z=i}=(0,0), \ \ i=0,1.
\end{array}
\right.
\end{equation}
The remainders in \eqref{3.18} are given as
\begin{equation}\label{3.20}
\begin{array}{lll}
D_1=&D-\sqrt{\epsilon}\left[(\psi(z))^2h^0_1\partial_x\eta^0_2+\psi(z)\eta^0_1\partial_x H^0_2+(\psi(z))^2\eta^0_1\partial_xh^0_2\right]\\
&-\sqrt{\epsilon}\big[(\psi(1-z))^2h^{u,0}_1\partial_x\eta^{u,0}_2
+\psi(1-z)\eta^{u,0}_1\partial_x H^0_2\\
&+(\psi(1-z))^2\eta^{u,0}_1\partial_xh^{u,0}_2+\psi(z)H^0_1\partial_x\eta^0_2
+\psi(1-z)H^0_1\partial_x\eta^{u,0}_2\big],
\end{array}
\end{equation}
\begin{equation}\label{3.21}
\begin{array}{lll}
E_1=E-\varepsilon\left[(\psi(z))^2\eta^0_1\partial_x \eta^{0}_2+(\psi(1-z))^2\eta^{u,0}_1\partial_x \eta^{u,0}_2\right],
\end{array}
\end{equation}
\begin{equation}\label{3.22}
\begin{array}{lll}
F_1=F-\sqrt{\varepsilon}\left[\psi(z)\partial_Z^2\eta^0_1+\psi(1-z)\partial_{Z^u}^2\eta^{u,0}_1\right],
\end{array}
\end{equation}
\begin{equation}\label{3.23}
\begin{array}{lll}
G_1=G-2\varepsilon\left[\psi'(z)\partial_Z\eta^0_1-\psi'(1-z)\partial_{Z^u}\eta^{u,0}_1\right],
\end{array}
\end{equation}
\begin{equation}\label{3.24}
\begin{array}{lll}
G_2=-\varepsilon^{\frac{3}{2}}\left[\psi''(z)\eta^0_1+\psi''(1-z)\eta^{u,0}_1\right],
\end{array}
\end{equation}
\begin{equation}\label{3.25}
\begin{array}{lll}
I_1=&I+\sqrt{\varepsilon}\big[\psi(z)\partial_t\eta^0_2+\psi(1-z)\partial_t\eta^{u,0}_2
-\psi(z)\partial_Z^2\eta^0_2-\psi(1-z)\partial_{Z^u}^2\eta^{u,0}_2\\
&+\psi(z)u^0_1\partial_x\eta^0_2+(\psi(z))^2\theta^0_1\partial_x\eta^0_2
-\psi(z)\eta^0_1\partial_xu^0_2-(\psi(z))^2\eta^0_1\partial_x \theta^0_2\\
&+\psi(1-z)u^0_1\partial_x\eta^{u,0}_2+(\psi(1-z))^2\theta^{u,0}_1\partial_x\eta^{u,0}_2\\
&-\psi(1-z)\eta^{u,0}_1\partial_xu^0_2-(\psi(1-z))^2\eta^{u,0}_1\partial_x \theta^{u,0}_2\big],
\end{array}
\end{equation}
\begin{equation}\label{3.26}
\begin{array}{lll}
J_1=&J+\varepsilon\big[-2\psi'(z)\partial_Z\eta^0_2-2\psi''(1-z)\partial_{Z^u}\eta^{u,0}_2\big],
\end{array}
\end{equation}
\begin{equation}\label{3.27}
\begin{array}{lll}
J_2=&-\varepsilon^{\frac{3}{2}}\big[\psi''(z)\eta^0_2+\psi''(1-z)\eta^{u,0}_2+\psi(z)\partial_{xx}\eta^0_2+\psi(1-z)\partial_{xx}\eta^{u,0}_2\big],
\end{array}
\end{equation}
in which $A,B,C,D,E,F,G,H,I,J$ are defined as in (\ref{3.5})--(\ref{3.14}).

The terms of approximate solutions (\ref{3.17}) are determined by the ideal MHD equations (\ref{1.5}), problems (\ref{2.3}) and (\ref{2.4}), therefore the approximate solutions are well-defined.

\section{The convergence rates : Proofs of the Theorem \ref{thm1}}\label{pf}
In this section, we will prove our main theorem. With the arguments stated in the previous sections, we know that the approximate solutions are well-defined.

To obtain the explicit convergence rate between the viscous solutions and the approximate solutions, we introduce the error solutions as follows
$$(u^{err},h^{err}):=(u^\varepsilon-\tilde{u}^a,H^\varepsilon-\tilde{h}^a),$$
then the equations for $(u^{err},h^{err})$ read  as
\begin{equation}\label{4.1}
\left \{
\begin{array}{lll}
\partial_t u^{err}_1-\varepsilon \partial_{zz} u^{err}_1=-(A+B),\\
\partial_t u^{err}_2-\varepsilon\Delta_{x,z}u^{err}_2+u^{err}_1\partial_x \tilde{u}^a_2+u^\varepsilon_1\partial_x u^{err}_2\\
\quad \quad \quad \quad -h^{err}_1\partial_x \tilde{h}^a_2-H^\varepsilon_1\partial_x h^{err}_2=-(C+D_1+E_1),\\
\partial_t h^{err}_1-\varepsilon \partial_{zz} h^{err}_1=-(F_1+G_1+G_2),\\
\partial_t h^{err}_2-\varepsilon\Delta_{x,z}h^{err}_2+u^{err}_1\partial_x \tilde{h}^a_2+u^\varepsilon_1\partial_x h^{err}_2\\
\quad \quad \quad \quad -h^{err}_1\partial_x \tilde{u}^a_2-H^\varepsilon_1\partial_x u^{err}_2=-(H+I_1+J_1+J_2),\\
(u^{err}_1,u^{err}_2, \partial_zh^{err}_1, \partial_zh^{err}_2)|_{z=i}=(0,0,0,0),\ i=0,1,\\
(u^{err}_1,u^{err}_2, h^{err}_1,h^{err}_2)|_{t=0}=(0,0,0,0),
\end{array}
\right.
\end{equation}
where the remainders are defined as in (\ref{3.20})--(\ref{3.27}).

Since the well-posedness of the viscous MHD, ideal MHD are classical, and the approximate solutions are well-defined, therefore we only need to deduce the convergence rates for the error solutions.

Before proving our main result, we introduce the anisotropic Sobolev inequality that will be used in the proof of our main theorems. See \cite{Mazzucato} for instance.
\begin{lemma}\label{lemma4.1}{\bf (\cite{Mazzucato})}
There holds that
\begin{equation}\label{4.4}
\begin{aligned}
\Vert u \Vert_{L^\infty((0,T) \times \Omega)} &\leq C \big(\Vert u \Vert_{L^\infty(0,T;L^2(\Omega))}^{\frac{1}{2}}\Vert \partial_z u \Vert_{L^\infty(0,T;L^2(\Omega))}^{\frac{1}{2}}\\
&\quad +\Vert \partial_z u \Vert_{L^\infty(0,T;L^2(\Omega))}^{\frac{1}{2}}\Vert \partial_x u \Vert_{L^\infty(0,T;L^2(\Omega))}^{\frac{1}{2}} \\
&\quad + \Vert u \Vert_{L^\infty(0,T;L^2(\Omega))}^{\frac{1}{2}}\Vert \partial_x \partial_z u \Vert_{L^\infty(0,T;L^2(\Omega))}^{\frac{1}{2}}\big),
\end{aligned}
\end{equation}
for all $u \in H^1_0(\Omega)$. It is pointed out that the left-hand sides of the inequality could be infinite.
\end{lemma}

Now we are on a position to prove Theorem \ref{thm1}.

\begin{proof}[Proof of Theorem \ref{thm1}]
We will complete our proof by the following several steps.

\textbf{Step 1: Estimates for $(u^{err}_1,h^{err}_1)$.}

Multiplying (\ref{4.1})$_1$ by $u^{err}_1$, integrating by parts over the $\Omega$, we have
\begin{equation}\label{4.2}
\begin{aligned}
\frac{1}{2}\frac{\mathrm{d}}{\mathrm{d}t}\Vert u^{err}_1\Vert_{L^2(0,1)}^2+\varepsilon\Vert \partial_z u^{err}_1\Vert_{L^2(0,1)}^2=-\int_0^1 (A+B)u^{err}_1\mathrm{d}z.
\end{aligned}
\end{equation}
As an example, one term of the right-hand side can be bounded by
\begin{equation}\nonumber
\begin{aligned}
\left|\int_0^1 \psi'(z)\partial_Z \theta^0_1(t;z/\sqrt{\varepsilon}) u^{err}_1 \right| &\leq \int_{\frac{1}{3}}^{\frac{2}{3}}\left|\partial_Z\theta^0_1 u^{err}_1 \right|\\
&\leq C \varepsilon^{\frac{5}{4}} \Vert u^{err}_1\Vert_{L^2(0,1)}\Vert \langle Z \rangle^2\partial_Z \theta^0_1\Vert_{L^2(0,\infty)},
\end{aligned}
\end{equation}
in which the weighted estimates for boundary layer correctors have been used (see \cite{Mazzucato} for instance) and the limits of integration are due to the support properties of the cut-off function $\psi$. Other terms can be estimated in a similar way. Therefore, we get that
\begin{equation}\label{4.3}
\begin{aligned}
\frac{1}{2}\frac{\mathrm{d}}{\mathrm{d}t}&\Vert u^{err}_1\Vert_{L^2(0,1)}^2+\varepsilon\Vert \partial_z u^{err}_1\Vert_{L^2(0,1)}^2\\
\leq &C \varepsilon^{\frac{7}{4}}\Vert u^{err}_1\Vert_{L^2}\left(\Vert \langle Z \rangle^2\partial_Z \theta^0_1\Vert_{L^2(0,\infty)}+\Vert \langle Z^u \rangle^2\partial_{Z^u} \theta^{u,0}_1\Vert_{L^2(0,\infty)}\right)\\
&+C\varepsilon \Vert u^{err}_1\Vert_{L^2}\left(\Vert  \theta^0_1\Vert_{L^2(0,\infty)}+\Vert \theta^{u,0}_1\Vert_{L^2(0,\infty)}+\Vert u^0_1\Vert_{H^2}\right).
\end{aligned}
\end{equation}
Applying the Cauchy inequality and Gronwall's inequality, we have
\begin{equation}\label{4.4}
\Vert u^{err}_1\Vert_{L^\infty(0,T;L^2(0,1))}+
\sqrt{\varepsilon}\Vert \partial_z u^{err}_1\Vert_{L^2(0,T;L^2(0,1))} \leq C \varepsilon.
\end{equation}

Multiplying (\ref{4.1})$_1$ by $-\partial_{zz} u^{err}_1$, integrating by parts over the $\Omega$ to yield that
\begin{equation}\label{4.5}
\begin{aligned}
\frac{1}{2}\frac{\mathrm{d}}{\mathrm{d}t}&\Vert \partial_z  u^{err}_1\Vert_{L^2(0,1)}^2+\varepsilon\Vert \partial_{zz} u^{err}_1\Vert_{L^2(0,1)}^2\\
\leq &C \varepsilon\Vert \partial_{zz} u^{err}_1\Vert_{L^2}\bigg(\Vert \langle Z \rangle^2\partial_Z \theta^0_1\Vert_{L^2(0,\infty)}+\Vert \langle Z^u \rangle^2\partial_{Z^u} \theta^{u,0}_1\Vert_{L^2(0,\infty)}\\
&+\Vert  \theta^0_1\Vert_{L^2(0,\infty)}+\Vert \theta^{u,0}_1\Vert_{L^2(0,\infty)}+\Vert u^0_1\Vert_{H^2}\bigg),
\end{aligned}
\end{equation}
then one has
\begin{equation}\label{4.6}
\Vert \partial_z u^{err}_1\Vert_{L^\infty(0.T;L^2(0,1))}+
\sqrt{\varepsilon}\Vert \partial_{zz} u^{err}_1\Vert_{L^2(0,T;L^2(0,1))} \leq C \sqrt{\varepsilon},
\end{equation}
where we have used the estimates obtained in Appendix of \cite{Mazzucato}.

Therefore, we have
\begin{equation}\label{4.7}
\Vert  u^{err}_1\Vert_{L^\infty(0,T;L^2(0,1))}\leq C \varepsilon,
\end{equation}
\begin{equation}\label{4.8}
\Vert  u^{err}_1\Vert_{L^\infty(0,T;H^1(0,1))}\leq C \sqrt{\varepsilon},
\end{equation}
\begin{equation}\label{4.9}
\begin{aligned}
\Vert  u^{err}_1\Vert_{L^\infty((0,T)\times (0,1))}\leq &\Vert  u^{err}_1\Vert_{L^\infty(0,T;L^2(0,1))}^{\frac{1}{2}}\Vert  u^{err}_1\Vert_{L^\infty(0,T;H^1(0,1))}^{\frac{1}{2}}\\
\leq &C \varepsilon^{\frac{3}{4}}.
\end{aligned}
\end{equation}

Following the similar arguments for $h^{err}_1$, we have
\begin{equation}\label{4.10}
\Vert  h^{err}_1\Vert_{L^\infty(0,T;L^2(0,1))}\leq C \varepsilon,
\end{equation}
\begin{equation}\label{4.11}
\Vert  h^{err}_1\Vert_{L^\infty(0,T;H^1(0,1))}\leq C \sqrt{\varepsilon},
\end{equation}
\begin{equation}\label{4.12}
\begin{aligned}
\Vert  h^{err}_1\Vert_{L^\infty((0,T)\times (0,1))}\leq C \varepsilon^{\frac{3}{4}}.
\end{aligned}
\end{equation}

\textbf{Step 2: Estimates for $(u^{err}_2,h^{err}_2)$.}

Multiplying (\ref{4.1})$_{2,4}$ by $u^{err}_2,h^{err}_2$, respectively, integrating by parts and adding the results, we get that
\begin{equation}\label{4.13}
\begin{aligned}
\frac{1}{2}\frac{\mathrm{d}}{\mathrm{d}t}\Vert (u^{err}_2,h^{err}_2)\Vert_{L^2}^2&+\varepsilon\Vert \nabla_{x,z}(u^{err}_2,h^{err}_2)\Vert_{L^2}^2=-\int_\Omega u^{err}_1\partial_x\tilde{u}^a_2u^{err}_2\\
&-\int_\Omega h^{err}_1u^{err}_2 \partial_x \tilde{h}^a_2-\int_\Omega (C+D_1+E_1)u^{err}_2\\
&-\int_\Omega u^{err}_1\partial_x \tilde{h}^a_2\cdot h^{err}_2-\int_\Omega h^{err}_1\partial_x \tilde{u}^a_2 \cdot h^{err}_2\\
&-\int_\Omega (H+I_1+J_1+J_2)h^{err}_2:=\sum_{i=1}^{11}I_i.
\end{aligned}
\end{equation}

Every term ($I=1,2,\cdots,10$) can be estimated as follows.
\begin{equation}\label{4.14}
\begin{aligned}
I_1 & \leq \Vert u^{err}_1 \Vert_{L^2}\left(\Vert \partial_{x}u^0_2\Vert_{L^\infty}+\Vert \partial_{x}\theta^0_2\Vert_{L^\infty}+\Vert \partial_{x}\theta^{u,0}_2\Vert_{L^\infty}\right)\Vert u^{err}_2 \Vert_{L^2}\\
&\leq C\varepsilon \left(\Vert \partial_{x}u^0_2\Vert_{L^\infty}+\Vert \partial_{x}\theta^0_2\Vert_{L^\infty}+\Vert \partial_{x}\theta^{u,0}_2\Vert_{L^\infty}\right)\Vert u^{err}_2 \Vert_{L^2},
\end{aligned}
\end{equation}
\begin{equation}\label{4.15}
\begin{aligned}
I_2 & \leq \Vert h^{err}_1 \Vert_{L^2}\left(\Vert \partial_{x}H^0_2\Vert_{L^\infty}+\Vert \partial_{x}h^0_2\Vert_{L^\infty}+\Vert \partial_{x}h^{u,0}_2\Vert_{L^\infty}\right)\Vert u^{err}_2 \Vert_{L^2}\\
&\leq C\varepsilon \left(\Vert \partial_{x}H^0_2\Vert_{L^\infty}+\Vert \partial_{x}h^0_2\Vert_{L^\infty}+\Vert \partial_{x}h^{u,0}_2\Vert_{L^\infty}\right)\Vert u^{err}_2 \Vert_{L^2},
\end{aligned}
\end{equation}
\begin{equation}\label{4.16}
\begin{aligned}
I_3  \leq &C\varepsilon^{\frac{5}{4}}\Vert u^{err}_2\Vert_{L^2}\bigg( \Vert \theta^0_1\Vert_{L^\infty}\Vert \langle Z\rangle^2\partial_x\theta^0_2\Vert_{L^2}+\Vert h^0_1\Vert_{L^\infty}\Vert \langle Z\rangle^2\partial_xh^0_2 \Vert_{L^2}\\
&+\Vert \theta^{u,0}_1\Vert_{L^\infty}\Vert \langle Z^u\rangle^2\partial_x\theta^{u,0}_2\Vert_{L^2}+\Vert h^{u,0}_1\Vert_{L^\infty}\Vert \langle Z^u\rangle^2\partial_xh^{u,0}_2 \Vert_{L^2}\bigg),
\end{aligned}
\end{equation}

\begin{equation}\label{4.17}
\begin{aligned}
I_4 \leq &C\varepsilon^{\frac{3}{4}}\Vert u^{err}_2\Vert_{L^2}\bigg( \Vert \partial_z u^0_1\Vert_{L^\infty}\Vert \langle Z\rangle \partial_x \theta^0_2\Vert_{L^2}+\Vert \langle Z\rangle \theta^0_1\Vert_{L^2}\Vert \partial_{zx}u^0_2\Vert_{L^\infty}\\
&+\Vert \partial_z H^0_1\Vert_{L^\infty}\Vert \langle Z\rangle \partial_xh^0_2\Vert_{L^2}+\Vert \partial_{zx}H^0_2\Vert_{L^\infty}\Vert \langle Z\rangle h^0_1\Vert_{L^2}\\
&+\Vert \partial_z u^0_1\Vert_{L^\infty}\Vert \langle Z^u\rangle \partial_x \theta^{u,0}_2\Vert_{L^2}+\Vert \langle Z^u\rangle \theta^{u,0}_1\Vert_{L^2}\Vert \partial_{zx}u^0_2\Vert_{L^\infty}\\
&+\Vert \partial_z H^0_1\Vert_{L^\infty}\Vert \langle Z^u\rangle \partial_xh^{u,0}_2\Vert_{L^2}+\Vert \partial_{zx}H^0_2\Vert_{L^\infty}\Vert \langle Z^u\rangle h^{u,0}_1\Vert_{L^2}\\
&+\Vert \partial_Z \theta^0_2\Vert_{L^2}+\Vert \partial_{Z^u} \theta^{u,0}_2\Vert_{L^2}+\|h^0_1\|_{L^\infty}\|\partial_x\eta^0_2\|_{L^2}+\|\eta^0_1\|_{L^2}\|\partial_x H^0_2\|_{L^\infty}\\
&+\|\eta^0_1\|_{L^\infty}\|\partial_xh^0_2\|_{L^2}
+\|h^{u,0}_1\|_{L^\infty}\|\partial_x\eta^{u,0}_2\|_{L^2}+\|\eta^{u,0}_1\|_{L^2}\|\partial_x H^0_2\|_{L^\infty}\\
&+\|\eta^{u,0}_1\|_{L^\infty}\|\partial_xh^{u,0}_2\|_{L^2}
+\|H^0_1\|_{L^\infty}\|\partial_x\eta^0_2\|_{L^2}+\|H^0_1\|_{L^\infty}\|\partial_x\eta^{u,0}_2\|_{L^2}\bigg),
\end{aligned}
\end{equation}
\begin{equation}\label{4.18}
\begin{aligned}
I_5  \leq &C\varepsilon\Vert u^{err}_2\Vert_{L^2} \bigg(\Vert u^0_2\Vert_{H^2}+\Vert \partial_{xx}\theta^0_2\Vert_{L^2}+\Vert \partial_{xx}\theta^{u,0}_2\Vert_{L^2}+\Vert \theta^0_2\Vert_{L^2}\\
&+\Vert\theta^{u,0}_2\Vert_{L^2}+\|\eta^0_1\|_{L^\infty}\|\partial_x \eta^{0}_2\|_{L^2}+\|\eta^{u,0}_1\|_{L^\infty}\|\partial_x \eta^{u,0}_2\|_{L^2}\bigg),
\end{aligned}
\end{equation}
\begin{equation}\label{4.19}
\begin{aligned}
I_6 & \leq \Vert u^{err}_1 \Vert_{L^2}\left(\Vert \partial_{x}H^0_2\Vert_{L^\infty}+\Vert \partial_{x}h^0_2\Vert_{L^\infty}+\Vert \partial_{x}h^{u,0}_2\Vert_{L^\infty}\right)\Vert h^{err}_2 \Vert_{L^2}\\
&\leq C\varepsilon \left(\Vert \partial_{x}H^0_2\Vert_{L^\infty}+\Vert \partial_{x}h^0_2\Vert_{L^\infty}+\Vert \partial_{x}h^{u,0}_2\Vert_{L^\infty}\right)\Vert h^{err}_2 \Vert_{L^2},
\end{aligned}
\end{equation}
\begin{equation}\label{4.20}
\begin{aligned}
I_7 & \leq \Vert h^{err}_1 \Vert_{L^2}\left(\Vert \partial_{x}u^0_2\Vert_{L^\infty}+\Vert \partial_{x}\theta^0_2\Vert_{L^\infty}+\Vert \partial_{x}\theta^{u,0}_2\Vert_{L^\infty}\right)\Vert h^{err}_2 \Vert_{L^2}\\
&\leq C\varepsilon \left(\Vert \partial_{x}u^0_2\Vert_{L^\infty}+\Vert \partial_{x}\theta^0_2\Vert_{L^\infty}+\Vert \partial_{x}\theta^{u,0}_2\Vert_{L^\infty}\right)\Vert h^{err}_2 \Vert_{L^2},
\end{aligned}
\end{equation}
\begin{equation}\label{4.21}
\begin{aligned}
I_8  \leq &C\varepsilon^{\frac{5}{4}}\Vert h^{err}_2\Vert_{L^2}\bigg( \Vert \theta^0_1\Vert_{L^\infty}\Vert \langle Z\rangle^2\partial_xh^0_2\Vert_{L^2}+\Vert h^0_1\Vert_{L^\infty}\Vert \langle Z\rangle^2\partial_x\theta^0_2 \Vert_{L^2}\\
&+\Vert \theta^{u,0}_1\Vert_{L^\infty}\Vert \langle Z^u\rangle^2\partial_xh^{u,0}_2\Vert_{L^2}+\Vert h^{u,0}_1\Vert_{L^\infty}\Vert \langle Z^u\rangle^2\partial_x\theta^{u,0}_2 \Vert_{L^2}\bigg),
\end{aligned}
\end{equation}
\begin{equation}\label{4.22}
\begin{aligned}
I_9 \leq &C\varepsilon^{\frac{3}{4}}\Vert h^{err}_2\Vert_{L^2}\bigg(\Vert\partial_z u^0_1\Vert_{L^\infty}\Vert \langle Z\rangle\partial_x h^0_2\Vert_{L^2}+\Vert \langle Z\rangle \theta^0_1\Vert_{L^2}\Vert \partial_{zx}H^0_2\Vert_{L^2}\\
&+\Vert\partial_z H^0_1\Vert_{L^\infty}\Vert \langle Z\rangle \partial_x \theta^0_2\Vert_{L^2}+\Vert\partial_{zx}H^0_2\Vert_{L^\infty}\Vert \langle Z\rangle h^0_1\Vert_{L^2}\\
&+\Vert\partial_z u^0_1\Vert_{L^\infty}\Vert \langle Z^u\rangle \partial_x h^{u,0}_2\Vert_{L^2}+\Vert \langle Z^u\rangle \theta^{u,0}_1\Vert_{L^2}\Vert \partial_{zx}H^0_2\Vert_{L^\infty}\\
&+\Vert \langle Z^u\rangle\partial_x \theta^{u,0}_2\Vert_{L^2}\Vert\partial_z H^0_1\Vert_{L^\infty}+\Vert\partial_{zx}H^0_2\Vert_{L^\infty}\Vert\langle Z^u\rangle\theta^{u,0}_1\Vert_{L^2}\\
&+\Vert\partial_Zh^0_2\Vert_{L^2}+\Vert\partial_{Z^u}h^{u,0}_2\Vert_{L^2}+\|\partial_t\eta^0_2\|_{L^2}+\|\partial_t\eta^{u,0}_2\|_{L^2}
+\|\partial_Z^2\eta^0_2\|_{L^2}\\
&+\|\partial_{Z^u}^2\eta^{u,0}_2\|_{L^2}+\|u^0_1\|_{L^\infty}\|\partial_x\eta^0_2\|_{L^2}+\|\theta^0_1\|_{L^\infty}\|\partial_x\eta^0_2\|_{L^2}\\
&+\|\eta^0_1\|_{L^2}\|\partial_xu^0_2\|_{L^\infty}+\|\eta^0_1\|_{L^\infty}\|\partial_x \theta^0_2\|_{L^2}+\|u^0_1\|_{L^\infty}\|\partial_x\eta^{u,0}_2\|_{L^2}\\
&+\|\theta^{u,0}_1\|_{L^\infty}\|\partial_x\eta^{u,0}_2\|_{L^2}+\|\eta^{u,0}_1\|_{L^2}\|\partial_xu^0_2\|_{L^\infty}+\|\eta^{u,0}_1\|_{L^\infty}\|\partial_x \theta^{u,0}_2\|_{L^2}\bigg),
\end{aligned}
\end{equation}
\begin{equation}\label{4.23}
\begin{aligned}
I_{10}  \leq &C\varepsilon\Vert h^{err}_2\Vert_{L^2} \bigg(\Vert H^0_2\Vert_{H^2}+\Vert \partial_{xx}h^0_2\Vert_{L^2}+\Vert\partial_{xx}h^{u,0}_2\Vert_{L^2}+\Vert h^0_2\Vert_{L^2}\\
&+\Vert h^{u,0}_2\Vert_{L^2}+\|\partial_Z\eta^0_2\|_{L^2}+\|\partial_{Z^u}\eta^{u,0}_2\|_{L^2}
\bigg),
\end{aligned}
\end{equation}
\begin{equation}\label{4.23a}
\begin{aligned}
I_{11}  \leq C\varepsilon^{\frac{3}{2}}\Vert h^{err}_2\Vert_{L^2} \bigg(\|\eta^0_2\|_{L^2}+\|\eta^{u,0}_2\|_{L^2}+\|\partial_{xx}\eta^0_2\|_{L^2}+\|\partial_{xx}\eta^{u,0}_2\|_{L^2}
\bigg).
\end{aligned}
\end{equation}
Putting the above estimates into (\ref{4.13}), applying the Cauchy's inequality and the Gronwall's inequality, we have
\begin{equation}\label{4.24}
\Vert (u^{err}_2,h^{err}_2)\Vert_{L^\infty(0,T;L^2(\Omega))}+\frac{\sqrt{\varepsilon}}{2}\Vert \nabla_{x,z}(u^{err}_2,h^{err}_2)\Vert_{L^2(0,T;L^2(\Omega))} \leq C \varepsilon^{\frac{3}{4}}.
\end{equation}

Multiplying (\ref{4.1})$_{2,4}$ by $-\partial_{xx} u^{err}_2,-\partial_{xx} h^{err}_2$, respectively, integrating on $\Omega$ and adding the results to give that
\begin{equation}\label{4.25}
\begin{aligned}
\frac{1}{2}\frac{\mathrm{d}}{\mathrm{d}t}\Vert \partial_x (u^{err}_2,h^{err}_2)\Vert_{L^2}^2&+\varepsilon\Vert \nabla_{x,z}\partial_x (u^{err}_2,h^{err}_2)\Vert_{L^2}^2=\int_\Omega u^{err}_1\partial_{xx}\tilde{u}^a_2\partial_x u^{err}_2\\
&+\int_\Omega h^{err}_1\partial_x u^{err}_2 \partial_{xx} \tilde{h}^a_2-\int_\Omega \partial_x (C+D_1+E_1)\partial_x u^{err}_2\\
&+\int_\Omega u^{err}_1\partial_{xx} \tilde{h}^a_2\partial_x h^{err}_2+\int_\Omega h^{err}_1\partial_{xx} \tilde{u}^a_2\partial_x h^{err}_2\\
&-\int_\Omega \partial_x (H+I_1+J_1+J_2)\partial_x h^{err}_2:=\sum_{i=1}^{11}M_i.
\end{aligned}
\end{equation}

Every term $M_i$ can be bounded by
\begin{equation}\label{4.26}
\begin{aligned}
M_1 & \leq \Vert u^{err}_1 \Vert_{L^2}\left(\Vert \partial_{xx}u^0_2\Vert_{L^\infty}+\Vert \partial_{xx}\theta^0_2\Vert_{L^\infty}+\Vert \partial_{xx}\theta^{u,0}_2\Vert_{L^\infty}\right)\Vert \partial_x u^{err}_2 \Vert_{L^2}\\
&\leq C\varepsilon \left(\Vert \partial_{xx}u^0_2\Vert_{L^\infty}+\Vert \partial_{xx}\theta^0_2\Vert_{L^\infty}+\Vert \partial_{xx}\theta^{u,0}_2\Vert_{L^\infty}\right)\Vert \partial_x u^{err}_2 \Vert_{L^2},
\end{aligned}
\end{equation}
\begin{equation}\label{4.27}
\begin{aligned}
M_2 & \leq \Vert h^{err}_1 \Vert_{L^2}\left(\Vert \partial_{xx}H^0_2\Vert_{L^\infty}+\Vert \partial_{xx}h^0_2\Vert_{L^\infty}+\Vert \partial_{xx}h^{u,0}_2\Vert_{L^\infty}\right)\Vert \partial_x u^{err}_2 \Vert_{L^2}\\
&\leq C\varepsilon \left(\Vert \partial_{xx}H^0_2\Vert_{L^\infty}+\Vert \partial_{xx}h^0_2\Vert_{L^\infty}+\Vert \partial_{xx}h^{u,0}_2\Vert_{L^\infty}\right)\Vert \partial_x u^{err}_2 \Vert_{L^2},
\end{aligned}
\end{equation}
\begin{equation}\label{4.28}
\begin{aligned}
M_3  \leq &C\varepsilon^{\frac{5}{4}}\Vert \partial_x u^{err}_2\Vert_{L^2}\bigg( \Vert \theta^0_1\Vert_{L^\infty}\Vert \langle Z\rangle^2\partial_{xx}\theta^0_2\Vert_{L^2}+\Vert h^0_1\Vert_{L^\infty}\Vert \langle Z\rangle^2\partial_{xx} h^0_2 \Vert_{L^2}\\
&+\Vert \theta^{u,0}_1\Vert_{L^\infty}\Vert \langle Z^u\rangle^2\partial_{xx}\theta^{u,0}_2\Vert_{L^2}+\Vert h^{u,0}_1\Vert_{L^\infty}\Vert \langle Z^u\rangle^2\partial_{xx} h^{u,0}_2 \Vert_{L^2}\bigg),
\end{aligned}
\end{equation}
\begin{equation}\label{4.29}
\begin{aligned}
M_4 \leq &C\varepsilon^{\frac{3}{4}}\Vert\partial_x u^{err}_2\Vert_{L^2}\bigg( \Vert \partial_z u^0_1\Vert_{L^\infty}\Vert \langle Z\rangle \partial_{xx} \theta^0_2\Vert_{L^2}+\Vert \langle Z\rangle \theta^0_1\Vert_{L^2}\Vert \partial_{zxx}u^0_2\Vert_{L^\infty}\\
&+\Vert \partial_z H^0_1\Vert_{L^\infty}\Vert \langle Z\rangle \partial_{xx} h^0_2\Vert_{L^2}+\Vert \partial_{zxx}H^0_2\Vert_{L^\infty}\Vert \langle Z\rangle h^0_1\Vert_{L^2}\\
&+\Vert \partial_z u^0_1\Vert_{L^\infty}\Vert \langle Z^u\rangle \partial_{xx} \theta^{u,0}_2\Vert_{L^2}+\Vert \langle Z^u\rangle \theta^{u,0}_1\Vert_{L^2}\Vert \partial_{zxx}u^0_2\Vert_{L^\infty}\\
&+\Vert \partial_z H^0_1\Vert_{L^\infty}\Vert \langle Z^u\rangle \partial_{xx} h^{u,0}_2\Vert_{L^2}+\Vert \partial_{zxx}H^0_2\Vert_{L^\infty}\Vert \langle Z^u\rangle h^{u,0}_1\Vert_{L^2}\\
&+\Vert \partial_{xZ} \theta^0_2\Vert_{L^2}+\Vert \partial_{xZ^u} \theta^{u,0}_2\Vert_{L^2}+\|h^0_1\|_{L^\infty}\|\partial_{xx}\eta^0_2\|_{L^2}\\
&+\|\eta^0_1\|_{L^2}\|\partial_{xx} H^0_2\|_{L^\infty}+\|\eta^0_1\|_{L^\infty}\|\partial_{xx}h^0_2\|_{L^2}+\|h^{u,0}_1\|_{L^\infty}\|\partial_{xx}\eta^{u,0}_2\|_{L^2}\\
&+\|\eta^{u,0}_1\|_{L^2}\|\partial_{xx} H^0_2\|_{L^\infty}+\|\eta^{u,0}_1\|_{L^\infty}\|\partial_{xx}h^{u,0}_2\|_{L^2}\\
&++\|H^0_1\|_{L^\infty}\|\partial_{xx}\eta^0_2\|_{L^2}+\|H^0_1\|_{L^\infty}\|\partial_{xx}\eta^{u,0}_2\|_{L^2}
\bigg),
\end{aligned}
\end{equation}
\begin{equation}\label{4.30}
\begin{aligned}
M_5  \leq &C\varepsilon\Vert\partial_x u^{err}_2\Vert_{L^2} \bigg(\Vert u^0_2\Vert_{H^3}+\Vert \partial_{xxx}\theta^0_2\Vert_{L^2}+\Vert \partial_{xxx}\theta^{u,0}_2\Vert_{L^2}+\Vert\partial_x \theta^0_2\Vert_{L^2}\\
&+\Vert\partial_x \theta^{u,0}_2\Vert_{L^2}
+\|\eta^0_1\|_{L^\infty}\|\partial_{xx} \eta^{0}_2\|_{L^2}+\|\eta^{u,0}_1\|_{L^\infty}\|\partial_{xx} \eta^{u,0}_2\|_{L^2}\bigg),
\end{aligned}
\end{equation}
\begin{equation}\label{4.31}
\begin{aligned}
M_6 & \leq \Vert u^{err}_1 \Vert_{L^2}\left(\Vert \partial_{xx}H^0_2\Vert_{L^\infty}+\Vert \partial_{xx}h^0_2\Vert_{L^\infty}+\Vert \partial_{xx}h^{u,0}_2\Vert_{L^\infty}\right)\Vert \partial_x h^{err}_2 \Vert_{L^2}\\
&\leq C\varepsilon \left(\Vert \partial_{xx}H^0_2\Vert_{L^\infty}+\Vert \partial_{xx}h^0_2\Vert_{L^\infty}+\Vert \partial_{xx}h^{u,0}_2\Vert_{L^\infty}\right)\Vert \partial_xh^{err}_2 \Vert_{L^2},
\end{aligned}
\end{equation}
\begin{equation}\label{4.32}
\begin{aligned}
M_7 & \leq \Vert h^{err}_1 \Vert_{L^2}\left(\Vert \partial_{xx}u^0_2\Vert_{L^\infty}+\Vert \partial_{xx}\theta^0_2\Vert_{L^\infty}+\Vert \partial_{xx}\theta^{u,0}_2\Vert_{L^\infty}\right)\Vert \partial_x h^{err}_2 \Vert_{L^2}\\
&\leq C\varepsilon \left(\Vert \partial_{xx}u^0_2\Vert_{L^\infty}+\Vert \partial_{xx}\theta^0_2\Vert_{L^\infty}+\Vert \partial_{xx}\theta^{u,0}_2\Vert_{L^\infty}\right)\Vert \partial_x h^{err}_2 \Vert_{L^2},
\end{aligned}
\end{equation}
\begin{equation}\label{4.33}
\begin{aligned}
M_8  \leq &C\varepsilon^{\frac{5}{4}}\Vert \partial_x h^{err}_2\Vert_{L^2}\bigg( \Vert \theta^0_1\Vert_{L^\infty}\Vert \langle Z\rangle^2\partial_{xx}h^0_2\Vert_{L^2}+\Vert h^0_1\Vert_{L^\infty}\Vert \langle Z\rangle^2\partial_{xx} \theta^0_2 \Vert_{L^2}\\
&+\Vert \theta^{u,0}_1\Vert_{L^\infty}\Vert \langle Z^u\rangle^2\partial_{xx}h^{u,0}_2\Vert_{L^2}+\Vert h^{u,0}_1\Vert_{L^\infty}\Vert \langle Z^u\rangle^2\partial_{xx}\theta^{u,0}_2 \Vert_{L^2}\bigg),
\end{aligned}
\end{equation}
\begin{equation}\label{4.34}
\begin{aligned}
M_9 \leq &C\varepsilon^{\frac{3}{4}}\Vert\partial_x h^{err}_2\Vert_{L^2}\bigg(\Vert\partial_z u^0_1\Vert_{L^\infty}\Vert \langle Z\rangle\partial_{xx} h^0_2\Vert_{L^2}+\Vert \langle Z\rangle \theta^0_1\Vert_{L^2}\Vert \partial_{zxx}H^0_2\Vert_{L^2}\\
&+\Vert\partial_z H^0_1\Vert_{L^\infty}\Vert \langle Z\rangle \partial_{xx} \theta^0_2\Vert_{L^2}+\Vert\partial_{zxx}H^0_2\Vert_{L^\infty}\Vert \langle Z\rangle h^0_1\Vert_{L^2}\\
&+\Vert\partial_z u^0_1\Vert_{L^\infty}\Vert \langle Z^u\rangle \partial_{xx} h^{u,0}_2\Vert_{L^2}+\Vert \langle Z^u\rangle \theta^{u,0}_1\Vert_{L^2}\Vert \partial_{zxx}H^0_2\Vert_{L^\infty}\\
&+\Vert \langle Z^u\rangle\partial_{xx} \theta^{u,0}_2\Vert_{L^2}\Vert\partial_z H^0_1\Vert_{L^\infty}+\Vert\partial_{zxx}H^0_2\Vert_{L^\infty}\Vert\langle Z^u\rangle\theta^{u,0}_1\Vert_{L^2}\\
&+\Vert\partial_{xZ}h^0_2\Vert_{L^2}+\Vert\partial_{xZ^u}h^{u,0}_2\Vert_{L^2}+\|\partial_{tx}\eta^0_2\|_{L^2}+\|\partial_{tx}\eta^{u,0}_2\|_{L^2}
+\|\partial_{xZZ}\eta^0_2\|_{L^2}\\
&+\|\partial_{xZ^uZ^u}\eta^{u,0}_2\|_{L^2}+\|u^0_1\|_{L^\infty}\|\partial_{xx}\eta^0_2\|_{L^2}+\|\theta^0_1\|_{L^\infty}\|\partial_{xx}\eta^0_2\|_{L^2}\\
&+\|\eta^0_1\|_{L^2}\|\partial_{xx}u^0_2\|_{L^\infty}+\|\eta^0_1\|_{L^\infty}\|\partial_{xx} \theta^0_2\|_{L^2}+\|u^0_1\|_{L^\infty}\|\partial_{xx}\eta^{u,0}_2\|_{L^2}\\
&+\|\theta^{u,0}_1\|_{L^\infty}\|\partial_{xx}\eta^{u,0}_2\|_{L^2}+\|\eta^{u,0}_1\|_{L^2}\|\partial_{xx}u^0_2\|_{L^\infty}+\|\eta^{u,0}_1\|_{L^\infty}\|\partial_{xx} \theta^{u,0}_2\|_{L^2}
\bigg),
\end{aligned}
\end{equation}
\begin{equation}\label{4.35}
\begin{aligned}
M_{10}  \leq &C\varepsilon\Vert \partial_x h^{err}_2\Vert_{L^2} \bigg(\Vert H^0_2\Vert_{H^3}+\Vert \partial_{xxx}h^0_2\Vert_{L^2}+\Vert\partial_{xxx}h^{u,0}_2\Vert_{L^2}\\
&+\Vert \partial_x h^0_2\Vert_{L^2}+\Vert \partial_x h^{u,0}_2\Vert_{L^2}
+\|\partial_{xZ}\eta^0_2\|_{L^2}+\|\partial_{xZ^u}\eta^{u,0}_2\|_{L^2}
\bigg),
\end{aligned}
\end{equation}
\begin{equation}\label{4.35a}
\begin{aligned}
M_{11}  \leq C\varepsilon^{\frac{3}{2}}\Vert \partial_x h^{err}_2\Vert_{L^2} \bigg(\|\partial_x\eta^0_2\|_{L^2}+\|\partial_x\eta^{u,0}_2++\|\partial_{xxx}\eta^0_2\|_{L^2}+\|\partial_{xxx}\eta^{u,0}_2\|_{L^2}\|_{L^2}
\bigg).
\end{aligned}
\end{equation}
Putting (\ref{4.26})-(\ref{4.35a}) into (\ref{4.25}), applying the Cauchy's inequality and Gronwall's inequality to yield that
\begin{equation}\label{4.36}
\Vert \partial_x (u^{err}_2,h^{err}_2)\Vert_{L^\infty(0,T;L^2(\Omega))}+\frac{\sqrt{\varepsilon}}{2}\Vert \nabla_{x,z}\partial_x (u^{err}_2,h^{err}_2)\Vert_{L^2(0,T;L^2(\Omega))} \leq C \varepsilon^{\frac{3}{4}}.
\end{equation}

Multiplying (\ref{4.1})$_{2,4}$ by $-\partial_{zz} u^{err}_2,-\partial_{zz} h^{err}_2$, respectively, integrating on $\Omega$ and adding the results to give that
\begin{equation}\label{4.37}
\begin{aligned}
\frac{1}{2}\frac{\mathrm{d}}{\mathrm{d}t}\Vert \partial_z (u^{err}_2,h^{err}_2)\Vert_{L^2}^2&+\varepsilon\Vert \nabla_{x,z}\partial_z (u^{err}_2,h^{err}_2)\Vert_{L^2}^2=\int_\Omega u^{err}_1\partial_x\tilde{u}^a_2\partial_{zz}u^{err}_2\\
&+\int_\Omega h^{err}_1\partial_{zz}u^{err}_2 \partial_x \tilde{h}^a_2+\int_\Omega (C+D_1+E_1)\partial_{zz}u^{err}_2\\
&+\int_\Omega u^{err}_1\partial_x \tilde{h}^a_2\partial_{zz}h^{err}_2+\int_\Omega h^{err}_1\partial_x \tilde{u}^a_2 \partial_{zz}h^{err}_2\\
&+\int_\Omega (H+I_1+J_1+J_2)\partial_{zz}h^{err}_2:=\sum_{i=1}^{11}K_i.
\end{aligned}
\end{equation}

We bound every $K_i(i=1,2,\cdots,10)$ as follows.
\begin{equation}\label{4.38}
\begin{aligned}
K_1 & \leq \Vert u^{err}_1 \Vert_{L^2}\left(\Vert \partial_{x}u^0_2\Vert_{L^\infty}+\Vert \partial_{x}\theta^0_2\Vert_{L^\infty}+\Vert \partial_{x}\theta^{u,0}_2\Vert_{L^\infty}\right)\Vert \partial_{zz}u^{err}_2 \Vert_{L^2}\\
&\leq C\varepsilon \left(\Vert \partial_{x}u^0_2\Vert_{L^\infty}+\Vert \partial_{x}\theta^0_2\Vert_{L^\infty}+\Vert \partial_{x}\theta^{u,0}_2\Vert_{L^\infty}\right)\Vert  \partial_{zz}u^{err}_2 \Vert_{L^2},
\end{aligned}
\end{equation}
\begin{equation}\label{4.39}
\begin{aligned}
K_2 & \leq \Vert h^{err}_1 \Vert_{L^2}\left(\Vert \partial_{x}H^0_2\Vert_{L^\infty}+\Vert \partial_{x}h^0_2\Vert_{L^\infty}+\Vert \partial_{x}h^{u,0}_2\Vert_{L^\infty}\right)\Vert  \partial_{zz}u^{err}_2 \Vert_{L^2}\\
&\leq C\varepsilon \left(\Vert \partial_{x}H^0_2\Vert_{L^\infty}+\Vert \partial_{x}h^0_2\Vert_{L^\infty}+\Vert \partial_{x}h^{u,0}_2\Vert_{L^\infty}\right)\Vert  \partial_{zz}u^{err}_2 \Vert_{L^2},
\end{aligned}
\end{equation}
\begin{equation}\label{4.40}
\begin{aligned}
K_3  \leq &C\varepsilon^{\frac{5}{4}}\Vert \partial_{zz} u^{err}_2\Vert_{L^2}\bigg( \Vert \theta^0_1\Vert_{L^\infty}\Vert \langle Z\rangle^2\partial_x\theta^0_2\Vert_{L^2}+\Vert h^0_1\Vert_{L^\infty}\Vert \langle Z\rangle^2\partial_xh^0_2 \Vert_{L^2}\\
&+\Vert \theta^{u,0}_1\Vert_{L^\infty}\Vert \langle Z^u\rangle^2\partial_x\theta^{u,0}_2\Vert_{L^2}+\Vert h^{u,0}_1\Vert_{L^\infty}\Vert \langle Z^u\rangle^2\partial_xh^{u,0}_2 \Vert_{L^2}\bigg),
\end{aligned}
\end{equation}

\begin{equation}\label{4.41}
\begin{aligned}
K_4 \leq &C\varepsilon^{\frac{3}{4}}\Vert\partial_{zz}u^{err}_2\Vert_{L^2}\bigg( \Vert \partial_z u^0_1\Vert_{L^\infty}\Vert \langle Z\rangle \partial_x \theta^0_2\Vert_{L^2}+\Vert \langle Z\rangle \theta^0_1\Vert_{L^2}\Vert \partial_{zx}u^0_2\Vert_{L^\infty}\\
&+\Vert \partial_z H^0_1\Vert_{L^\infty}\Vert \langle Z\rangle \partial_xh^0_2\Vert_{L^2}+\Vert \partial_{zx}H^0_2\Vert_{L^\infty}\Vert \langle Z\rangle h^0_1\Vert_{L^2}\\
&+\Vert \partial_z u^0_1\Vert_{L^\infty}\Vert \langle Z^u\rangle \partial_x \theta^{u,0}_2\Vert_{L^2}+\Vert \langle Z^u\rangle \theta^{u,0}_1\Vert_{L^2}\Vert \partial_{zx}u^0_2\Vert_{L^\infty}\\
&+\Vert \partial_z H^0_1\Vert_{L^\infty}\Vert \langle Z^u\rangle \partial_xh^{u,0}_2\Vert_{L^2}+\Vert \partial_{zx}H^0_2\Vert_{L^\infty}\Vert \langle Z^u\rangle h^{u,0}_1\Vert_{L^2}\\
&+\Vert \partial_Z \theta^0_2\Vert_{L^2}+\Vert \partial_{Z^u} \theta^{u,0}_2\Vert_{L^2}+\|h^0_1\|_{L^\infty}\|\partial_x\eta^0_2\|_{L^2}+\|\eta^0_1\|_{L^2}\|\partial_x H^0_2\|_{L^\infty}\\
&+\|\eta^0_1\|_{L^\infty}\|\partial_xh^0_2\|_{L^2}
+\|h^{u,0}_1\|_{L^\infty}\|\partial_x\eta^{u,0}_2\|_{L^2}+\|\eta^{u,0}_1\|_{L^2}\|\partial_x H^0_2\|_{L^\infty}\\
&+\|\eta^{u,0}_1\|_{L^\infty}\|\partial_xh^{u,0}_2\|_{L^2}
+\|H^0_1\|_{L^\infty}\|\partial_x\eta^0_2\|_{L^2}+\|H^0_1\|_{L^\infty}\|\partial_x\eta^{u,0}_2\|_{L^2}\bigg),
\end{aligned}
\end{equation}
\begin{equation}\label{4.42}
\begin{aligned}
K_5  \leq &C\varepsilon\Vert \partial_{zz}u^{err}_2\Vert_{L^2} \bigg(\Vert u^0_2\Vert_{H^2}+\Vert \partial_{xx}\theta^0_2\Vert_{L^2}+\Vert \partial_{xx}\theta^{u,0}_2\Vert_{L^2}+\Vert \theta^0_2\Vert_{L^2}\\
&+\Vert\theta^{u,0}_2\Vert_{L^2}
+\Vert\theta^{u,0}_2\Vert_{L^2}+\|\eta^0_1\|_{L^\infty}\|\partial_x \eta^{0}_2\|_{L^2}+\|\eta^{u,0}_1\|_{L^\infty}\|\partial_x \eta^{u,0}_2\|_{L^2}\bigg),
\end{aligned}
\end{equation}
\begin{equation}\label{4.43}
\begin{aligned}
K_6 & \leq \Vert u^{err}_1 \Vert_{L^2}\left(\Vert \partial_{x}H^0_2\Vert_{L^\infty}+\Vert \partial_{x}h^0_2\Vert_{L^\infty}+\Vert \partial_{x}h^{u,0}_2\Vert_{L^\infty}\right)\Vert h^{err}_2 \Vert_{L^2}\\
&\leq C\varepsilon \left(\Vert \partial_{x}H^0_2\Vert_{L^\infty}+\Vert \partial_{x}h^0_2\Vert_{L^\infty}+\Vert \partial_{x}h^{u,0}_2\Vert_{L^\infty}\right)\Vert  \partial_{zz}h^{err}_2 \Vert_{L^2},
\end{aligned}
\end{equation}
\begin{equation}\label{4.44}
\begin{aligned}
K_7 & \leq \Vert h^{err}_1 \Vert_{L^2}\left(\Vert \partial_{x}u^0_2\Vert_{L^\infty}+\Vert \partial_{x}\theta^0_2\Vert_{L^\infty}+\Vert \partial_{x}\theta^{u,0}_2\Vert_{L^\infty}\right)\Vert h^{err}_2 \Vert_{L^2}\\
&\leq C\varepsilon \left(\Vert \partial_{x}u^0_2\Vert_{L^\infty}+\Vert \partial_{x}\theta^0_2\Vert_{L^\infty}+\Vert \partial_{x}\theta^{u,0}_2\Vert_{L^\infty}\right)\Vert  \partial_{zz}h^{err}_2 \Vert_{L^2},
\end{aligned}
\end{equation}
\begin{equation}\label{4.45}
\begin{aligned}
K_8  \leq &C\varepsilon^{\frac{5}{4}}\Vert \partial_{zz} h^{err}_2\Vert_{L^2}\bigg( \Vert \theta^0_1\Vert_{L^\infty}\Vert \langle Z\rangle^2\partial_xh^0_2\Vert_{L^2}+\Vert h^0_1\Vert_{L^\infty}\Vert \langle Z\rangle^2\partial_x\theta^0_2 \Vert_{L^2}\\
&+\Vert \theta^{u,0}_1\Vert_{L^\infty}\Vert \langle Z^u\rangle^2\partial_xh^{u,0}_2\Vert_{L^2}+\Vert h^{u,0}_1\Vert_{L^\infty}\Vert \langle Z^u\rangle^2\partial_x\theta^{u,0}_2 \Vert_{L^2}\bigg),
\end{aligned}
\end{equation}
\begin{equation}\label{4.46}
\begin{aligned}
K_9 \leq &C\varepsilon^{\frac{3}{4}}\Vert \partial_{zz}h^{err}_2\Vert_{L^2}\bigg(\Vert\partial_z u^0_1\Vert_{L^\infty}\Vert \langle Z\rangle\partial_x h^0_2\Vert_{L^2}+\Vert \langle Z\rangle \theta^0_1\Vert_{L^2}\Vert \partial_{zx}H^0_2\Vert_{L^2}\\
&+\Vert\partial_z H^0_1\Vert_{L^\infty}\Vert \langle Z\rangle \partial_x \theta^0_2\Vert_{L^2}+\Vert\partial_{zx}H^0_2\Vert_{L^\infty}\Vert \langle Z\rangle h^0_1\Vert_{L^2}\\
&+\Vert\partial_z u^0_1\Vert_{L^\infty}\Vert \langle Z^u\rangle \partial_x h^{u,0}_2\Vert_{L^2}+\Vert \langle Z^u\rangle \theta^{u,0}_1\Vert_{L^2}\Vert \partial_{zx}H^0_2\Vert_{L^\infty}\\
&+\Vert \langle Z^u\rangle\partial_x \theta^{u,0}_2\Vert_{L^2}\Vert\partial_z H^0_1\Vert_{L^\infty}+\Vert\partial_{zx}H^0_2\Vert_{L^\infty}\Vert\langle Z^u\rangle\theta^{u,0}_1\Vert_{L^2}\\
&+\Vert\partial_Zh^0_2\Vert_{L^2}+\Vert\partial_{Z^u}h^{u,0}_2\Vert_{L^2}+\|\partial_t\eta^0_2\|_{L^2}+\|\partial_t\eta^{u,0}_2\|_{L^2}
+\|\partial_Z^2\eta^0_2\|_{L^2}\\
&+\|\partial_{Z^u}^2\eta^{u,0}_2\|_{L^2}+\|u^0_1\|_{L^\infty}\|\partial_x\eta^0_2\|_{L^2}+\|\theta^0_1\|_{L^\infty}\|\partial_x\eta^0_2\|_{L^2}\\
&+\|\eta^0_1\|_{L^2}\|\partial_xu^0_2\|_{L^\infty}+\|\eta^0_1\|_{L^\infty}\|\partial_x \theta^0_2\|_{L^2}+\|u^0_1\|_{L^\infty}\|\partial_x\eta^{u,0}_2\|_{L^2}\\
&+\|\theta^{u,0}_1\|_{L^\infty}\|\partial_x\eta^{u,0}_2\|_{L^2}+\|\eta^{u,0}_1\|_{L^2}\|\partial_xu^0_2\|_{L^\infty}+\|\eta^{u,0}_1\|_{L^\infty}\|\partial_x \theta^{u,0}_2\|_{L^2}\bigg),
\end{aligned}
\end{equation}
\begin{equation}\label{4.47}
\begin{aligned}
K_{10}  \leq &C\varepsilon\Vert  \partial_{zz}h^{err}_2\Vert_{L^2} \bigg(\Vert H^0_2\Vert_{H^2}+\Vert \partial_{xx}h^0_2\Vert_{L^2}
+\Vert\partial_{xx}h^{u,0}_2\Vert_{L^2}+\Vert h^0_2\Vert_{L^2}\\
&+\Vert h^{u,0}_2\Vert_{L^2}+\|\partial_Z\eta^0_2\|_{L^2}+\|\partial_{Z^u}\eta^{u,0}_2\|_{L^2}
\bigg),
\end{aligned}
\end{equation}
\begin{equation}\label{4.47a}
\begin{aligned}
K_{11}  \leq C\varepsilon^{\frac{3}{2}}\Vert\partial_{zz} h^{err}_2\Vert_{L^2} \bigg(\|\eta^0_2\|_{L^2}+\|\eta^{u,0}_2\|_{L^2}+\|\partial_{xx}\eta^0_2\|_{L^2}
+\|\partial_{xx}\eta^{u,0}_2\|_{L^2}
\bigg).
\end{aligned}
\end{equation}
Putting (\ref{4.38})-(\ref{4.47a}) into (\ref{4.37}), using the Cauchy's inequality and Gronwall's inequality to get that
\begin{equation}\label{4.48}
\Vert \partial_z (u^{err}_2,h^{err}_2)\Vert_{L^\infty(0,T;L^2(\Omega))}+\frac{\sqrt{\varepsilon}}{2}\Vert \nabla_{x,z}\partial_z (u^{err}_2,h^{err}_2)\Vert_{L^2(0,T;L^2(\Omega))} \leq C \varepsilon^{\frac{1}{4}}.
\end{equation}

It should be pointed out that the above bounds of order $\varepsilon^{\frac{1}{4}}$ can not be improved since we can not using integration by parts in the right-hand side involving second or mixed derivatives in $z$, as $\partial_z u^{err}_2$ may not vanish on the boundaries. In addition, although one can apply the integration by parts in $z$ for the magnetic fields, the convergence rates can not be improved due to the loss of $\sqrt{\varepsilon}$ resulted from the derivatives in $z$ for the remainders, which will give the same convergence rates.

Following the similar arguments, we also have
\begin{equation}\label{4.49}
\begin{aligned}
\Vert \partial_{xx} (u^{err}_2,h^{err}_2) \Vert_{L^\infty(0,T;L^2(\Omega))} \leq C \varepsilon^{\frac{3}{4}}
\end{aligned}
\end{equation}
and
\begin{equation}\label{4.50}
\begin{aligned}
\Vert \partial_{z} \partial_x (u^{err}_2,h^{err}_2) \Vert_{L^\infty(0,T;L^2(\Omega))}\leq C \varepsilon^{\frac{1}{4}}.
\end{aligned}
\end{equation}

Therefore, we deduce that
\begin{equation}\label{4.51}
\begin{aligned}
\Vert (u^{err}_2,h^{err}_2) \Vert_{L^\infty(0,T;H^1(\Omega))} \leq C\varepsilon^{\frac{1}{4}}.
\end{aligned}
\end{equation}

Finally, we use the Lemma \ref{lemma4.1} to get
\begin{equation}\label{4.52}
\begin{aligned}
\Vert u^{err}_2 \Vert_{L^\infty((0,T) \times \Omega)} &\leq C \bigg(\Vert u^{err}_2 \Vert_{L^\infty(0,T;L^2(\Omega))}^{\frac{1}{2}}\Vert \partial_z u^{err}_2 \Vert_{L^\infty(0,T;L^2(\Omega))}^{\frac{1}{2}}\\
&\quad +\Vert \partial_z u^{err}_2 \Vert_{L^\infty(0,T;L^2(\Omega))}^{\frac{1}{2}}\Vert \partial_x u^{err}_2 \Vert_{L^\infty(0,T;L^2(\Omega))}^{\frac{1}{2}} \\
&\quad + \Vert u^{err}_2 \Vert_{L^\infty(0,T;L^2(\Omega))}^{\frac{1}{2}}\Vert \partial_x \partial_z u^{err}_2 \Vert_{L^\infty(0,T;L^2(\Omega))}^{\frac{1}{2}}\bigg)\\
&\leq C\sqrt{\varepsilon}.
\end{aligned}
\end{equation}
To obtain the $L^\infty$ estimates for $h^{err}_2$, we will replace the Lemma \ref{lemma4.1} by the following inequality (see also Lemma 3.6 in \cite{Xin1})
$$\|u(x,z)\|_{L^\infty_{x,z}}\lesssim\|u\|_{L^2_xL^2_z}^{\frac{1}{4}}\|u\|_{H^1_xL^2_z}^{\frac{1}{4}}\|u\|_{L^2_xH^1_z}^{\frac{1}{4}}
\|u\|_{H^1_xH^1_z}^{\frac{1}{4}}.$$
Therefore, one has
\begin{equation}\label{4.53}
\begin{aligned}
\Vert h^{err}_2 \Vert_{L^\infty((0,T) \times \Omega)}\leq C\sqrt{\varepsilon}.
\end{aligned}
\end{equation}

Combine the above steps, we have
\begin{equation}\label{4.56}
\begin{aligned}
\Vert (u^{err},h^{err}) \Vert_{L^\infty(0,T;L^2(\Omega))} &\leq C\varepsilon^\frac{3}{4},\\
\Vert (u^{err},h^{err}) \Vert_{L^\infty(0,T;H^1(\Omega))} &\leq C\varepsilon^\frac{1}{4},\\
\Vert (u^{err},h^{err}) \Vert_{L^\infty((0,T) \times \Omega)}& \leq C\sqrt{\varepsilon},
\end{aligned}
\end{equation}
which completes the proof of Theorem \ref{thm1}.
\end{proof}

We end this section with the following corollary.
\begin{corollary}\label{cor1}
Under the assumptions of Theorem \ref{thm1}, the following optimal convergence rate holds
\begin{equation}\label{1.77}
\begin{array}{lll}
C \varepsilon^{\frac{1}{4}} \leq \Vert (u^\varepsilon-u^0,H^\varepsilon-H^0) \Vert_{L^\infty(0,T;L^2(\Omega))} \leq C \varepsilon^{\frac{1}{4}},
\end{array}
\end{equation}
where $(u^0,H^0)$ is the solution of Problem (\ref{1.5})-(\ref{1.6}) and the constants $C>0$ are independent of $\varepsilon$.
\end{corollary}

\begin{proof}
The results of Corollary \ref{cor1} is straightforward from the fact that
$$\Vert (\theta^0,\theta^{u,0},h^0,h^{u,0},\widetilde{h^0},\widetilde{h^{u,0}}) \Vert_{L^\infty (0,T;L^2(\Omega_\infty))} \thickapprox \varepsilon^{\frac{1}{4}},$$
where the boundary layer correctors $\theta^0,\theta^{u,0},h^0,h^{u,0},\widetilde{h^0},\widetilde{h^{u,0}}$ are determined in Section \ref{approximate} and $\Omega_\infty=[0,L]\times [0,\infty)$.
\end{proof}

\section{The case with uniform magnetic background: stabilizing effects of magnetic fields}\label{stabilityeffect}
Based on the results obtained in Section \ref{pf}, we will study the stabilizing effects of magnetic fields. Precisely, we will study the case that the perfectly conducting wall condition for magnetic fields replaced by uniform magnetic background, i.e., the (nonhomogeneous) Dirichlet boundary condition, see \cite{Davidson,G-P,Gilbert,CLiu2,CLiu3} and the references therein. The boundary conditions and initial data in this case are stated as follows
\begin{equation}\label{1.3a}
\left \{
\begin{array}{lll}
u^\varepsilon|_{z=i}=\alpha^i(t;x), \ \ \alpha^i(t;x)=(\alpha^i_1(t),\alpha^i_2(t;x),0), \ \ i=0,1,\\
H^\varepsilon|_{z=i}=\gamma^i(t;x), \ \ \gamma^i(t;x)=(\gamma^i_1(t),\gamma^i_2(t;x),0), \ \ i=0,1,\\
(u^\varepsilon,H^\varepsilon)|_{t=0}=(u_0,H_0),\ u_0=(a(z),b(x,z),0),\ H_0=(c(z),d(x,z),0).
\end{array}
\right.
\end{equation}
With the boundary conditions, the zero-order and first-order compatibility conditions would be given as follows respectively:
\begin{equation}\label{1.4a0}
\left \{
\begin{array}{lll}
\alpha^i(0;x)=u_0(x,i), \ \ i=0,1,\\
\gamma^i(0;x)=H_0(x,i), \ \ i=0,1,\\
\end{array}
\right.
\end{equation}
and
\begin{equation}\label{1.4b0}
\left \{
\begin{array}{lll}
\partial_t \alpha^i_1(0)-\varepsilon \partial_{zz}a(i)=f_1(0;i),\\
\partial_t \alpha^i_2(0;x)-\varepsilon \Delta_{x,z}b(x,i)+a(i)\partial_x b(x,i)-c(i)\partial_x d(x,i)=f_2(0;x,i),\\
\partial_t \gamma^i_1(0)-\varepsilon \partial_{zz}c(i)=0,\\
\partial_t \gamma^i_2(0;x)-\varepsilon \Delta_{x,z}d(x,i)+a(i)\partial_x d(x,i)-c(i)\partial_x b(x,i)=0,
\end{array}
\right.
\end{equation}
for $i=0,1$.

It is well known that the no-slip type boundary conditions will result in the strong boundary layer. However, in the case with uniform magnetic background, compared with Theorem \ref{thm1}, the convergence rates do not be worse, whose reasons are mainly resulted from the stabilizing effects of the magnetic fields and the structure of the plane MHD flow.

Following the arguments as in Section \ref{boundarylayerequ}, we assume that the viscous MHD solutions are well approximated by
\begin{equation}\label{appnoslip}
\left \{
\begin{array}{lll}
u^a_1(t;z)&:=u^{ou}_1(t;z)+\theta^0_1\left(t;\frac{z}{\sqrt{\varepsilon}}\right)+\theta^{u,0}_1\left(t;\frac{1-z}{\sqrt{\varepsilon}}\right),\\
u^a_2(t;x,z)&:=u^{ou}_2(t;x,z)+\theta^0_2\left(t;x,\frac{z}{\sqrt{\varepsilon}}\right)+\theta^{u,0}_2\left(t;x,\frac{1-z}{\sqrt{\varepsilon}}\right),\\
h^a_1(t;z)&:=H^{ou}_1(t;z)+h^0_1\left(t;\frac{z}{\sqrt{\varepsilon}}\right)+h^{u,0}_1\left(t;\frac{1-z}{\sqrt{\varepsilon}}\right),\\
h^a_2(t;x,z)&:=H^{ou}_2(t;x,z)+h^0_2\left(t;x,\frac{z}{\sqrt{\varepsilon}}\right)+h^{u,0}_2\left(t;x,\frac{1-z}{\sqrt{\varepsilon}}\right),
\end{array}
\right.
\end{equation}
where the correctors satisfy that
\begin{equation}\label{appnoslip1}
(\theta^0_i,h^0_i) \to 0 \ \ \mathrm{as} \ Z \to \infty; \ \ \ (\theta^{u,0}_i,h^{u,0}_i) \to 0 \ \ \mathrm{as} \ Z^u \to \infty,
\end{equation}
in which $i=1,2, Z=z/\sqrt{\varepsilon}$ and $Z^u=(1-z)/\sqrt{\varepsilon}$.

Every part in the approximate solutions satisfy the following problems:

\textbf{(I) The outer solutions $(u^{ou},H^{ou})$.}

The outer solutions $(u^{ou},H^{ou})$ satisfy the ideal MHD equations (\ref{1.5}) with the initial data
\begin{equation}\label{2.2}
(u^0,H^0)|_{t=0}=(u_0,H_0).
\end{equation}
The uniqueness of the solutions to the system implies that $(u^{ou},H^{ou})\equiv (u^0,H^0)$.

\textbf{(II) The lower correctors $(\theta^0_1,\theta^0_2,h^0_1,h^0_2)$.}

The lower correctors $(\theta^0_1,\theta^0_2,h^0_1,h^0_2)$ satisfy that
\begin{equation}\label{appnoslip2}
\left \{
\begin{array}{lll}
\partial_t \theta^0_1-\partial_{ZZ}\theta^0_1=0,\\
\partial_t \theta^0_2-\partial_{ZZ}\theta^0_2+(u^0_1(t;0)+\theta^0_1)\partial_x \theta^0_2+\theta^0_1\partial_x u^0_2(t;x,0)\\
\quad\quad\quad  -(H^0_1(t;0)+h^0_1)\partial_x h^0_2-h^0_1\partial_x H^0_2(t;x,0)=0,\\
\partial_t h^0_1-\partial_{ZZ}h^0_1=0,\\
\partial_t h^0_2-\partial_{ZZ}h^0_2+(u^0_1(t;0)+\theta^0_1)\partial_x h^0_2+\theta^0_1\partial_x H^0_2(t;x,0)\\
\quad\quad\quad  -(H^0_1(t;0)+h^0_1)\partial_x \theta^0_2-h^0_1\partial_x u^0_2(t;x,0)=0,\\
(\theta^0_1,\theta^0_2)|_{Z=0}=(\alpha^0_1(t)-u^0_1(t;0),\alpha^0_2(t;x)-u^0_2(t;x,0)),\\
(h^0_1,  h^0_2)|_{Z=0}=(\gamma^0_1(t)-H^0_1(0),\gamma^0_2(t;x)-H^0_2(t;x,0)),\\
(\theta^0_1,\theta^0_2, h^0_1,h^0_2)|_{Z=\infty}=(0,0,0,0),\\
(\theta^0_1,\theta^0_2, h^0_1,h^0_2)|_{t=0}=(0,0,0,0).
\end{array}
\right.
\end{equation}

\textbf{(III) The upper correctors $(\theta^{u,0}_1,\theta^{u,0}_2,h^{u,0}_1,h^{u,0}_2)$.}

The lower correctors $(\theta^{u,0}_1,\theta^{u,0}_2,h^{u,0}_1,h^{u,0}_2)$ satisfy that
\begin{equation}\label{appnoslip3}
\left \{
\begin{array}{lll}
\partial_t \theta^{u,0}_1-\partial_{Z^uZ^u}\theta^{u,0}_1=0,\\
\partial_t \theta^{u,0}_2-\partial_{Z^uZ^u}\theta^{u,0}_2+(u^0_1(t;1)+\theta^{u,0}_1)\partial_x \theta^{u,0}_2+\theta^{u,0}_1\partial_x u^0_2(t;x,1)\\
\quad\quad\quad  -(H^0_1(t;1)+h^{u,0}_1)\partial_x h^{u,0}_2-h^{u,0}_1\partial_x h^{u,0}_2(t;x,1)=0,\\
\partial_t h^{u,0}_1-\partial_{Z^uZ^u}h^{u,0}_1=0,\\
\partial_t h^{u,0}_2-\partial_{Z^uZ^u}h^{u,0}_2+(u^0_1(t;1)+\theta^{u,0}_1)\partial_x h^{u,0}_2+\theta^{u,0}_1\partial_x h^{u,0}_2(t;x,1)\\
\quad\quad\quad  -(H^0_1(t;1)+h^{u,0}_1)\partial_x \theta^{u,0}_2-h^{u,0}_1\partial_x u^0_2(t;x,1)=0,\\
(\theta^{u,0}_1,\theta^{u,0}_2)|_{Z^u=0}=(\alpha^1_1(t)-u^1_1(t;1),\alpha^1_2(t;x)-u^0_2(t;x,1)),\\
(h^{u,0}_1, h^{u,0}_2)|_{Z^u=0}=(\gamma^1_1(t)-H^0_1(1),\gamma^1_2(t;x)-h^0_2(t;x,1)),\\
(\theta^{u,0}_1,\theta^{u,0}_2, h^{u,0}_1,h^{u,0}_2)|_{Z^u=\infty}=(0,0,0,0),\\
(\theta^{u,0}_1,\theta^{u,0}_2, h^{u,0}_1,h^{u,0}_2)|_{t=0}=(0,0,0,0).
\end{array}
\right.
\end{equation}

It is noted that the boundary conditions for both correctors of velocity and magnetic fields are Dirichlet boundary conditions, then the well-posedness and regularity results of the above \eqref{appnoslip2} and \eqref{appnoslip3} are also classical, see \cite{Evans,Xin1} for details. Therefore the approximate solutions in this case are well-defined.

Following the arguments as before, let $\psi(z)$ be a smooth function on $[0,1]$ with
\begin{equation}\label{cutoffnonsl}
\psi(z)=\left \{
\begin{array}{lll}
1,& z\in [0,\frac{1}{3}],\\
0,& z \in [\frac{1}{2},1],\\
\mathrm{smooth},& \mathrm{otherwise}.
\end{array}
\right.
\end{equation}
We introduce the truncated approximations as follows
\begin{equation}\label{appnoslip4}
\left \{
\begin{array}{lll}
\tilde{u}^a_1(t;z)&:=u^{0}_1(t;z)+\psi(z)\theta^0_1\left(t;\frac{z}{\sqrt{\varepsilon}}\right)+\psi(1-z)\theta^{u,0}_1\left(t;\frac{1-z}{\sqrt{\varepsilon}}\right),\\
\tilde{u}^a_2(t;x,z)&:=u^{0}_2(t;x,z)+\psi(z)\theta^0_2\left(t;x,\frac{z}{\sqrt{\varepsilon}}\right)+\psi(1-z)\theta^{u,0}_2\left(t;x,\frac{1-z}{\sqrt{\varepsilon}}\right),\\
\tilde{h}^a_1(t;z)&:=H^{0}_1(t;z)+\psi(z)h^0_1\left(t;\frac{z}{\sqrt{\varepsilon}}\right)+\psi(1-z)h^{u,0}_1\left(t;\frac{1-z}{\sqrt{\varepsilon}}\right),\\
\tilde{h}^a_2(t;x,z)&:=H^{0}_2(t;x,z)+\psi(z)h^0_2\left(t;x,\frac{z}{\sqrt{\varepsilon}}\right)+\psi(1-z)h^{u,0}_2\left(t;x,\frac{1-z}{\sqrt{\varepsilon}}\right),
\end{array}
\right.
\end{equation}
then the $(\tilde{u}^a,\tilde{H}^a)$ satisfy that
\begin{equation}\label{appnoslip5}
\left \{
\begin{array}{lll}
\partial_t \tilde{u}^a_1-\varepsilon \partial_{zz} \tilde{u}^a_1=f_1+A+B,\\
\partial_t \tilde{u}^a_2-\varepsilon \Delta_{x,z}\tilde{u}^a_2 +\tilde{u}^a_1\partial_x \tilde{u}^a_2-\tilde{h}^a_1\partial_x \tilde{h}^a_2=f_2+C+D+E,\\
\partial_t \tilde{h}^a_1-\varepsilon \partial_{zz} \tilde{h}^a_1=F+G,\\
\partial_t \tilde{h}^a_2-\varepsilon \Delta_{x,z}\tilde{h}^a_2 +\tilde{u}^a_1\partial_x \tilde{h}^a_2-\tilde{h}^a_1\partial_x \tilde{u}^a_2=H+I+J,\\
\end{array}
\right.
\end{equation}
with the following initial data and boundary conditions
\begin{equation}\label{appnoslip6}
\left \{
\begin{array}{lll}
(\tilde{u}^a,\tilde{h}^a)|_{t=0}=(u_0,H_0),\\
(\tilde{u}^a,\tilde{h}^a)|_{z=i}=(\alpha^i,\gamma^i)(t;x), \ \ i=0,1,
\end{array}
\right.
\end{equation}
where the remainders are given by (\ref{3.5})--(\ref{3.14}).
Introduce the error solutions as follows
$$(u^{err},h^{err}):=(u^\varepsilon-\tilde{u}^a,H^\varepsilon-\tilde{h}^a),$$
then the equations for $(u^{err},h^{err})$ read  as
\begin{equation}\label{appnoslip6}
\left \{
\begin{array}{lll}
\partial_t u^{err}_1-\varepsilon \partial_{zz} u^{err}_1=-(A+B),\\
\partial_t u^{err}_2-\varepsilon\Delta_{x,z}u^{err}_2+u^{err}_1\partial_x \tilde{u}^a_2+u^\varepsilon_1\partial_x u^{err}_2\\
\quad \quad \quad \quad -h^{err}_1\partial_x \tilde{h}^a_2-h^\varepsilon_1\partial_x h^{err}_2=-(C+D+E),\\
\partial_t h^{err}_1-\varepsilon \partial_{zz} h^{err}_1=-(F+G),\\
\partial_t h^{err}_2-\varepsilon\Delta_{x,z}h^{err}_2+u^{err}_1\partial_x \tilde{h}^a_2+u^\varepsilon_1\partial_x h^{err}_2\\
\quad \quad \quad \quad -h^{err}_1\partial_x \tilde{u}^a_2-h^\varepsilon_1\partial_x u^{err}_2=-(H+I+J),\\
(u^{err}_1,u^{err}_2, h^{err}_1, h^{err}_2)|_{z=i}=(0,0,0,0),\ i=0,1,\\
(u^{err}_1,u^{err}_2, h^{err}_1,h^{err}_2)|_{t=0}=(0,0,0,0),
\end{array}
\right.
\end{equation}
where the remainders $A,B,C,D,E,F,G,H,I,J$ are defined as in (\ref{3.5})--(\ref{3.14}).

For the case with uniform magnetic background, we have the following result.
\begin{theorem}\label{thm:nonslip}
Suppose the initial and boundary data, external force satisfy that $u_0, H_0 \in H^m(\Omega)$, $f \in L^\infty(0,T;H^m(\Omega))$, $ \alpha^i,\gamma^i \in H^2(0,T; H^m(\partial{\Omega})),i=0,1, m>5$ and the compatibility condition (\ref{1.4a0}). Then there exist positive constants $C>0$, independent of $\varepsilon$, such that for any solution $(u^\varepsilon,H^\varepsilon)$ of (\ref{1.4}) with the initial data $(u_0,H_0)$ and boundary data $\alpha^i,\gamma^i$ in \eqref{1.3a}, satisfying that
\begin{equation}\label{nonslip1}
\Vert (u^\varepsilon-\tilde{u}^a,H^\varepsilon-\tilde{h}^a) \Vert_{L^\infty(0,T;L^2(\Omega))} \leq C \varepsilon^{\frac{3}{4}},
\end{equation}
\begin{equation}\label{nonslip2}
\Vert (u^\varepsilon-\tilde{u}^a,H^\varepsilon-\tilde{h}^a) \Vert_{L^\infty(0,T;H^1(\Omega))} \leq C \varepsilon^{\frac{1}{4}},
\end{equation}
\begin{equation}\label{nonslip3}
\Vert (u^\varepsilon-\tilde{u}^a,H^\varepsilon-\tilde{h}^a)\Vert_{L^\infty((0,T) \times \Omega))} \leq C \sqrt{\varepsilon},
\end{equation}
where $\tilde{u}^a,\tilde{h}^a$ are defined by (\ref{appnoslip4}).
\end{theorem}
\begin{proof}
The proof is similar to the arguments as in Section \ref{pf}, the result follows. Here we omit the details.
\end{proof}
\begin{remark}
Compared with Theorem \ref{thm1}, the convergence rates in Theorem \ref{thm:nonslip} do not become worse, which imply the stabilizing effect of the magnetic fields and the good structure of the plane MHD flow.
\end{remark}
The optimal convergence rate result can be obtained similarly and here we omit it.

\section{Improved convergence rates}\label{improved}
In this section, we would use the higher-order expansions to improve the convergence rates obtained in Section \ref{pf} and Section \ref{stabilityeffect}. Note that the arguments for the case with perfectly conducting wall are similar to that with uniform magnetic background, therefore we only discuss the case with uniform magnetic background for simplicity.
We introduce
\begin{equation}\label{5.1}
\left \{
\begin{array}{lll}
u^{a,1}(t;x,z)=u^{ou}(t;x,z)+u^{lc}\left(t;x,\frac{z}{\sqrt{\varepsilon}}\right)+u^{uc}\left(t;x,\frac{1-z}{\sqrt{\varepsilon}}\right),\\
h^{a,1}(t;x,z)=H^{ou}(t;x,z)+H^{lc}\left(t;x,\frac{z}{\sqrt{\varepsilon}}\right)+H^{uc}\left(t;x,\frac{1-z}{\sqrt{\varepsilon}}\right),
\end{array}
\right.
\end{equation}
in which, the outer solution, lower corrector and upper corrector are defined as
\begin{equation}\label{5.2}
\left \{
\begin{array}{lll}
u^{\textrm{ou}}_1&:=u^0_1(t;z)+\sqrt{\varepsilon}u^1_1(t;z),\\
u^{\textrm{ou}}_2&:=u^0_2(t;x,z)+\sqrt{\varepsilon}u^1_2(t;x,z),\\
H^{\textrm{ou}}_1&:=H^0_1(z)+\sqrt{\varepsilon}H^1_1(t;z),\\
H^{\textrm{ou}}_2&:=H^0_2(t;x,z)+\sqrt{\varepsilon}H^1_2(t;x,z),
\end{array}
\right.
\end{equation}
\begin{equation}\label{5.3}
\left \{
\begin{array}{lll}
u^{\textrm{lc}}_1\left(t;\frac{z}{\sqrt{\varepsilon}}\right)&:=\theta^0_1\left(t;\frac{z}{\sqrt{\varepsilon}}\right)+\sqrt{\varepsilon}\theta^1_1\left(t;\frac{z}{\sqrt{\varepsilon}}\right),\\
u^{\textrm{lc}}_2\left(t;x,\frac{z}{\sqrt{\varepsilon}}\right)&:=\theta^0_2\left(t;x,\frac{z}{\sqrt{\varepsilon}}\right)+\sqrt{\varepsilon}\theta^1_2\left(t;x,\frac{z}{\sqrt{\varepsilon}}\right),\\
H^{\textrm{lc}}_1\left(t;\frac{z}{\sqrt{\varepsilon}}\right)&:=h^0_1\left(t;\frac{z}{\sqrt{\varepsilon}}\right)+\sqrt{\varepsilon}h^1_1\left(t;\frac{z}{\sqrt{\varepsilon}}\right),\\
H^{\textrm{lc}}_2\left(t;x,\frac{z}{\sqrt{\varepsilon}}\right)&:=h^0_2\left(t;x,\frac{z}{\sqrt{\varepsilon}}\right)+\sqrt{\varepsilon}h^1_2\left(t;x,\frac{z}{\sqrt{\varepsilon}}\right),
\end{array}
\right.
\end{equation}
and
\begin{equation}\label{5.4}
\left \{
\begin{array}{lll}
u^{\textrm{uc}}_1\left(t;\frac{1-z}{\sqrt{\varepsilon}}\right)&:=\theta^{u,0}_1\left(t;\frac{1-z}{\sqrt{\varepsilon}}\right)+\sqrt{\varepsilon}\theta^{u,1}_1\left(t;\frac{1-z}{\sqrt{\varepsilon}}\right),\\
u^{\textrm{uc}}_2\left(t;x,\frac{1-z}{\sqrt{\varepsilon}}\right)&:=\theta^{u,0}_2\left(t;x,\frac{1-z}{\sqrt{\varepsilon}}\right)+\sqrt{\varepsilon}\theta^{u,1}_2\left(t;x,\frac{1-z}{\sqrt{\varepsilon}}\right),\\
H^{\textrm{uc}}_1\left(t;\frac{1-z}{\sqrt{\varepsilon}}\right)&:=h^{u,0}_1\left(t;\frac{1-z}{\sqrt{\varepsilon}}\right)+\sqrt{\varepsilon}h^{u,1}_1\left(t;\frac{1-z}{\sqrt{\varepsilon}}\right),\\
H^{\textrm{uc}}_2\left(t;x,\frac{1-z}{\sqrt{\varepsilon}}\right)&:=h^{u,0}_2\left(t;x,\frac{1-z}{\sqrt{\varepsilon}}\right)+\sqrt{\varepsilon}h^{u,1}_2\left(t;x,\frac{1-z}{\sqrt{\varepsilon}}\right),
\end{array}
\right.
\end{equation}
the correctors should obey the matching conditions
\begin{equation}\label{5.5}
\left \{
\begin{array}{lll}
(\theta^i_1,\theta^{i}_2,h^i_1,h^{i}_2) \to (0,0,0,0) \ \ &\textrm{as} \ Z \to \infty,\\
(\theta^{u,i}_1,\theta^{u,i}_2,h^{u,i}_1,h^{u,i}_2)\to (0,0,0,0)  \ \ &\textrm{as} \ Z^u \to \infty,
\end{array}
\right.
\end{equation}
where $i=0,1$ and $Z=\frac{z}{\sqrt{\varepsilon}}$ and $Z^u=\frac{1-z}{\sqrt{\varepsilon}}$.

We would derive the equations satisfied by the outer solutions and the correctors. It is obvious that the terms of the leading order terms in the outer solutions and the correctors are determined by the ideal MHD solutions $(u^0,H^0)$ and $(\theta^0,h^0),(\theta^{u,0},h^{u,0})$ constructed in Section \ref{stabilityeffect}, respectively. The first-order terms are determined as follows:

(I) The first-order terms of the outer solutions $(u^1,H^1)$ satisfy the following equations:
\begin{equation}\label{5.6}
\left \{
\begin{array}{lll}
\partial_t u^1_1=0,\\
\partial_t u^1_2+u^0_1\partial_x u^1_2+u^1_1\partial_x u^0_1-H^0_1\partial_x H^1_2-H^1_1\partial_x H^0_2=0,\\
\partial_t H^1_1=0,\\
\partial_t H^1_2+u^0_1\partial_x H^1_2+u^1_1\partial_x H^0_1-H^0_1\partial_x u^1_2-H^1_1\partial_x u^0_2=0,\\
(u^1_1,u^1_2,H^1_1,H^1_2)|_{t=0}=(0,0,0,0).
\end{array}
\right.
\end{equation}

For given the ideal MHD solutions $(u^0,H^0)$ with enough regularity, it can be deduced that $(u^1_1,u^1_2,H^1_1,H^1_2)\equiv (0,0,0,0)$. Indeed, it follows from (\ref{5.6})$_{1,3}$ that $(u^1_1,H^1_1)\equiv (0,0)$. Therefore (\ref{5.6}) can be reduced to
\begin{equation}\label{5.7}
\left \{
\begin{array}{lll}
(u^1_1,H^1_1)\equiv (0,0),\\
\partial_t u^1_2+u^0_1\partial_x u^1_2-H^0_1\partial_x H^1_2=0,\\
\partial_t H^1_2+u^0_1\partial_x H^1_2-H^0_1\partial_x u^1_2=0,\\
(u^1_1,u^1_2,H^1_1,H^1_2)|_{t=0}=(0,0,0,0).
\end{array}
\right.
\end{equation}
Let $U=(u^1_2,H^1_2)$, it can rewritten as
\begin{equation}\label{5.8}
\left \{
\begin{array}{lll}
\partial_t U+A\partial_x U=0,\\
U|_{t=0}=(0,0),
\end{array}
\right.
\end{equation}
where $A$ is a regular enough matrix given by
\begin{gather*}
A=\begin{pmatrix} u^0_1  & -H^0_1 \\ -H^0_1 & u^0_1 \end{pmatrix}.
\end{gather*}
Then one can know from the theory of hyperbolic systems that $U\equiv (0,0)$.

(II) The first-order terms in the lower correctors $(\theta^1,h^1)$ satisfy
\begin{equation}\label{5.9}
\left \{
\begin{array}{lll}
\partial_t \theta^1_1-\partial_{ZZ}\theta^1_1=0,\\
\partial_t \theta^1_2-\partial_{ZZ}\theta^1_2+(u^0_1(t;0)+\theta^0_1)\partial_x\theta^1_2+\theta^1_1(\partial_x \theta^0_2+\partial_xu^0_2(t;x,0))\\
\quad -(H^0_1(0)+h^0_1)\partial_xh^1_2-h^1_1(\partial_x h^0_2+\partial_xH^0_2(t;x,0))\\
\quad =-Z\bigg(\theta^0_1\partial_{zx}u^0_2(t;x,0)+\partial_z u^0_1(t;0)\partial_x\theta^0_2\\
\quad\quad -h^0_1\partial_{zx}H^0_2(t;x,0)-\partial_z H^0_1(0)\partial_xh^0_2\bigg),\\
\partial_t h^1_1-\partial_{ZZ}h^1_1=0,\\
\partial_t h^1_2-\partial_{ZZ}h^1_2+(u^0_1(t;0)+\theta^0_1)\partial_x h^1_2+\theta^1_1(\partial_x h^0_2+\partial_xH^0_2(t;x,0))\\
\quad -(H^0_1(0)+h^0_1)\partial_xu^1_2-h^1_1(\partial_x \theta^0_2+\partial_xu^0_2(t;x,0))\\
\quad =-Z\bigg(\theta^0_1\partial_{zx}H^0_2(t;x,0)+\partial_z u^0_1(t;0)\partial_xh^0_2\\
\quad\quad -h^0_1\partial_{zx}u^0_2(t;x,0)-\partial_z H^0_1(0)\partial_x\theta^0_2\bigg),\\
(\theta^1_1,\theta^1_2, h^1_1, h^1_2)|_{Z=0}=(0,0,0,0),\\
(\theta^1_1,\theta^1_2,h^1_1,h^1_2)|_{Z=\infty}=(0,0,0,0),\\
(\theta^1_1,\theta^1_2,h^1_1,h^1_2)|_{t=0}=(0,0,0,0).
\end{array}
\right.
\end{equation}

By (\ref{5.9})$_{1,3}$, the boundary conditions and initial data, we deduce that
\begin{equation}\label{5.10}
(\theta^1_1,h^1_1)\equiv (0,0).
\end{equation}
Therefore, the system (\ref{5.9}) can be reduced to
\begin{equation}\label{5.11}
\left \{
\begin{array}{lll}
(\theta^1_1,h^1_1)=(0,0),\\
\partial_t \theta^1_2-\partial_{ZZ}\theta^1_2+(u^0_1(t;0)+\theta^0_1)\partial_x\theta^1_2 -(H^0_1(0)+h^0_1)\partial_xh^1_2\\
\quad =-Z\bigg(\theta^0_1\partial_{zx}u^0_2(t;x,0)+\partial_z u^0_1(t;0)\partial_x\theta^0_2\\
\quad\quad -h^0_1\partial_{zx}H^0_2(t;x,0)-\partial_z H^0_1(0)\partial_xh^0_2\bigg),\\
\partial_t h^1_2-\partial_{ZZ}h^1_2+(u^0_1(t;0)+\theta^0_1)\partial_x h^1_2-(H^0_1(0)+h^0_1)\partial_xu^1_2\\
\quad =-Z\bigg(\theta^0_1\partial_{zx}H^0_2(t;x,0)+\partial_z u^0_1(t;0)\partial_xh^0_2\\
\quad\quad -h^0_1\partial_{zx}u^0_2(t;x,0)-\partial_z H^0_1(0)\partial_x\theta^0_2\bigg),\\
(\theta^1_2,h^1_2)|_{Z=0}=(0,0),\\
(\theta^1_2,h^1_2)|_{Z=\infty}=(0,0),\\
(\theta^1_2,h^1_2)|_{t=0}=(0,0).
\end{array}
\right.
\end{equation}

(III) Similarly, one can deduce that first-order terms in the upper correctors $(\theta^{u,1},h^{u,1})$ satisfy
\begin{equation}\label{5.12}
\left \{
\begin{array}{lll}
(\theta^{u,1}_1,h^{u,1}_1)=(0,0),\\
\partial_t \theta^{u,1}_2-\partial_{Z^uZ^u}\theta^{u,1}_2+(u^0_1(t;1)+\theta^{u,0}_1)\partial_x\theta^{u,1}_2 -(H^0_1(1)+h^{u,0}_1)\partial_xh^{u,1}_2\\
\quad =-Z^u\bigg(\theta^{u,0}_1\partial_{zx}u^0_2(t;x,1)+\partial_z u^0_1(t;1)\partial_x\theta^{u,0}_2\\
\quad\quad -h^{u,0}_1\partial_{zx}H^0_2(t;x,1)-\partial_z H^0_1(1)\partial_xh^{u,0}_2\bigg),\\
\partial_t h^{u,1}_2-\partial_{Z^uZ^u}h^{u,1}_2+(u^0_1(t;1)+\theta^{u,0}_1)\partial_x h^{u,1}_2-(H^0_1(1)+h^{u,0}_1)\partial_xu^{u,1}_2\\
\quad =-Z^u\bigg(\theta^{u,0}_1\partial_{zx}H^0_2(t;x,1)+\partial_z u^0_1(t;1)\partial_xh^{u,0}_2\\
\quad\quad -h^{u,0}_1\partial_{zx}u^0_2(t;x,1)-\partial_z H^0_1(1)\partial_x\theta^{u,0}_2\bigg),\\
(\theta^{u,1}_2,h^{u,1}_2)|_{Z^u=0}=(0,0),\\
(\theta^{u,1}_2,h^{u,1}_2)|_{Z^u=\infty}=(0,0),\\
(\theta^{u,1}_2,h^{u,1}_2)|_{t=0}=(0,0).
\end{array}
\right.
\end{equation}

Note that the above problems (\ref{5.11}) and (\ref{5.12}) are linear, and their well-posedness theory with different boundary condition and weighted estimates can be obtained by the standard Picard iteration method, which are similar to that in Section \ref{boundarylayerequ}. The first-order compatibility conditions are used to improve the regularity of the correctors. Similar arguments can be applied for case with perfectly conducting condition. We omit the details of the proof and refer to \cite{Mazzucato,Xin1} for instance.

As before, we introduce the modified approximate solutions $(\tilde{u}^{a,1},\tilde{h}^{a,1})$ as follows
\begin{equation}\label{5.13}
\left \{
\begin{array}{lll}
\tilde{u}^{a,1}_1(t;z)&:=u^{0}_1(t;z)+\psi(z)\theta^0_1\left(t;\frac{z}{\sqrt{\varepsilon}}\right)+\psi(1-z)\theta^{u,0}_1\left(t;\frac{1-z}{\sqrt{\varepsilon}}\right),\\
\tilde{u}^{a,1}_2(t;x,z)&:=u^{0}_2(t;x,z)+\psi(z)\left(\theta^0_2\left(t;x,\frac{z}{\sqrt{\varepsilon}}\right)+\sqrt{\varepsilon}\theta^1_2\left(t;\frac{z}{\sqrt{\varepsilon}}\right)\right)\\
&\quad \quad +\psi(1-z)\left(\theta^{u,0}_2\left(t;x,\frac{1-z}{\sqrt{\varepsilon}}\right)+\sqrt{\varepsilon}\theta^{u,1}_2\left(t;x,\frac{1-z}{\sqrt{\varepsilon}}\right)\right),\\
\tilde{h}^{a,1}_1(t;z)&:=H^{0}_1(t;z)+\psi(z)h^0_1\left(t;\frac{z}{\sqrt{\varepsilon}}\right)+\psi(1-z)h^{u,0}_1\left(t;\frac{1-z}{\sqrt{\varepsilon}}\right),\\
\tilde{h}^{a,1}_2(t;x,z)&:=H^{0}_2(t;x,z)+\psi(z)\left(h^0_2\left(t;x,\frac{z}{\sqrt{\varepsilon}}\right)+\sqrt{\varepsilon}h^1_2\left(t;x,\frac{z}{\sqrt{\varepsilon}}\right)\right)\\
&\quad \quad +\psi(1-z)\left(h^{u,0}_2\left(t;x,\frac{1-z}{\sqrt{\varepsilon}}\right)+\sqrt{\varepsilon}h^{u,1}_2\left(t;x,\frac{1-z}{\sqrt{\varepsilon}}\right)\right),
\end{array}
\right.
\end{equation}
in which $\psi$ is the cut-off function use in Section \ref{approximate}. Here we note that $(u^1,H^1)=(0,0)$ and $(\theta^1_1,h^1_1,\theta^{u,1}_1,h^{u,1}_1)=(0,0,0,0)$, therefore $(\tilde{u}^{a,1}_1,\tilde{H}^{a,1}_1)=(\tilde{u}^{a}_1,\tilde{H}^{a}_1)$.

The modified approximate solutions satisfy the following equations
\begin{equation}\label{5.14}
\left \{
\begin{array}{lll}
\partial_t \tilde{u}^{a,1}_1-\varepsilon \partial_{zz} \tilde{u}^{a,1}_1=f_1+A+B,\\
\partial_t \tilde{u}^{a,1}_2-\varepsilon \Delta_{x,z}\tilde{u}^{a,1}_2 +\tilde{u}^{a,1}_1\partial_x \tilde{u}^{a,1}_2-\tilde{h}^{a,1}_1\partial_x \tilde{h}^{a,1}_2=f_2+C+\hat{D}+\hat{E}+\hat{M},\\
\partial_t \tilde{h}^{a,1}_1-\varepsilon \partial_{zz} \tilde{h}^{a,1}_1=F+G,\\
\partial_t \tilde{h}^{a,1}_2-\varepsilon \Delta_{x,z}\tilde{h}^{a,1}_2 +\tilde{u}^{a,1}_1\partial_x \tilde{h}^{a,1}_2-\tilde{h}^{a,1}_1\partial_x \tilde{u}^{a,1}_2=H+\hat{I}+\hat{J}+\hat{N},\\
\end{array}
\right.
\end{equation}
with the following initial and boundary conditions
\begin{equation}\label{5.15}
\left \{
\begin{array}{lll}
(\tilde{u}^{a,1},\tilde{h}^{a,1})|_{t=0}=(u_0,H_0),\\
(\tilde{u}^{a,1},\tilde{h}^{a,1})|_{z=i}=(\alpha^i,\gamma^i)(t;x), \ \ i=0,1,
\end{array}
\right.
\end{equation}
in which, the remainders $A,B,C,F,G,H$ are given as in Section \ref{approximate}, and the others are defined as follows:
\begin{equation}\label{5.16}
\begin{array}{lll}
\hat{D}=&\sqrt{\varepsilon}\Big[\psi(z)(\psi(z)-1)\left(\theta^0_1\partial_x \theta^1_2-h^0_1\partial_x h^1_2\right)\\
&+\psi(1-z)(\psi(1-z)-1)\left(\theta^{u,0}_1\partial_x \theta^{u,1}_2-h^{u,0}_1\partial_x h^{u,1}_2\right)\\
&-2\psi'(z)\partial_Z\theta^0_2-2\psi'(1-z)\partial_{Z^u}\theta^{u,0}_2\Big],
\end{array}
\end{equation}
\begin{equation}\label{5.17}
\begin{array}{lll}
\hat{E}=&\varepsilon\bigg[\psi(z)\bigg(Z\partial_x \theta^1_2\partial_z u^0_1(t;0)+\frac{1}{2}\partial_{zz}u^0_1(t;0)Z^2\partial_x \theta^0_2\\
&+\frac{1}{2}Z^2\theta^0_1\partial_{xzz}u^0_2(t;x,0)-Z\partial_x h^1_2\partial_z H^0_1(0)-\frac{1}{2}\partial_{zz}H^0_1(0)Z^2\partial_x h^0_2\\
&-\frac{1}{2}Z^2h^0_1\partial_{xzz}H^0_2(t;x,0)\bigg) \\
&+\psi(1-z)\bigg(-Z^u\partial_x \theta^{u,1}_2\partial_z u^0_1(t;1)+\frac{1}{2}\partial_{zz}u^0_1(t;1)(Z^u)^2\partial_x \theta^{u,0}_2\\
&+\frac{1}{2}(Z^u)^2\theta^{u,0}_1\partial_{xzz}u^0_2(t;x,1) +Z^u\partial_x h^{u,1}_2\partial_z H^0_1(1)\\
&-\frac{1}{2}\partial_{zz}H^0_1(1)(Z^u)^2\partial_x h^{u,0}_2-\frac{1}{2}(Z^u)^2h^{u,0}_1\partial_{xzz}H^0_2(t;x,1)\bigg)\\
&-\Delta_{x,z}u^0_2-\psi(z)\partial_{xx}\theta^0_2\\
&-\psi(1-z)\partial_{xx}\theta^{u,0}_2-\psi''(z)\theta^0_2-\psi''(1-z)\theta^{u,0}_2\bigg],
\end{array}
\end{equation}
\begin{equation}\label{5.18}
\begin{array}{lll}
\hat{M}=&\varepsilon^{\frac{3}{2}}\bigg[-\psi''(z)\theta^1_2-\psi''(1-z)\theta^{u,1}_2\\
&+\psi(z)\bigg(\frac{1}{2}\partial_{z}^2 u^0_1(t;0)Z^2\partial_x\theta^1_2-\partial_x^2\theta^1_2\bigg)\\
&+\psi(1-z)\bigg(\frac{1}{2}\partial_{z}^2 u^0_1(t;1)(Z^u)^2\partial_x\theta^{u,1}_2-\partial_x^2\theta^{u,1}_2\bigg)\\
&-\psi(z)\bigg(\frac{1}{2}\partial_{z}^2 H^0_1(0)Z^2\partial_xh^1_2-\partial_x^2h^1_2\bigg)\\
&-\psi(1-z)\bigg(\frac{1}{2}\partial_{z}^2 H^0_1(1)(Z^u)^2\partial_xh^{u,1}_2-\partial_x^2h^{u,1}_2\bigg)\bigg],
\end{array}
\end{equation}
\begin{equation}\label{5.19}
\begin{array}{lll}
\hat{I}=&\sqrt{\varepsilon}\Big[\psi(z)(\psi(z)-1)\left(\theta^0_1\partial_x h^1_2-h^0_1\partial_x \theta^1_2\right)\\
&+\psi(1-z)(\psi(1-z)-1)\left(\theta^{u,0}_1\partial_x h^{u,1}_2-h^{u,0}_1\partial_x \theta^{u,1}_2\right)\\
&-2\psi'(z)\partial_Zh^0_2-2\psi'(1-z)\partial_{Z^u}h^{u,0}_2\Big],
\end{array}
\end{equation}
\begin{equation}\label{5.20}
\begin{array}{lll}
\hat{J}=&\varepsilon\bigg[\psi(z)\bigg(Z\partial_x h^1_2\partial_z u^0_1(t;0)+\frac{1}{2}\partial_{zz}u^0_1(t;0)Z^2\partial_x h^0_2\\
&+\frac{1}{2}Z^2\theta^0_1\partial_{xzz}H^0_2(t;x,0)-Z\partial_x \theta^1_2\partial_z H^0_1(0)-\frac{1}{2}\partial_{zz}H^0_1(t;0)Z^2\partial_x \theta^0_2\\
&-\frac{1}{2}Z^2h^0_1\partial_{xzz}u^0_2(t;x,0)\bigg) \\
&+\psi(1-z)\bigg(-Z^u\partial_x h^{u,1}_2\partial_z u^0_1(t;1)+\frac{1}{2}\partial_{zz}u^0_1(t;1)Z^2\partial_x h^{u,0}_2\\
&+\frac{1}{2}(Z^u)^2\theta^{u,0}_1\partial_{xzz}H^0_2(t;x,1)+Z^u\partial_x \theta^{u,1}_2\partial_z H^0_1(1)\\
&-\frac{1}{2}\partial_{zz}H^0_1(1)(Z^u)^2\partial_x \theta^{u,0}_2-\frac{1}{2}(Z^u)^2h^{u,0}_1\partial_{xzz}u^0_2(t;x,1)\bigg)\\
&-\Delta_{x,z}H^0_2-\psi(z)\partial_{xx}h^0_2\\
&-\psi(1-z)\partial_{xx}h^{u,0}_2-\psi''(z)h^0_2-\psi''(1-z)h^{u,0}_2\bigg],
\end{array}
\end{equation}
\begin{equation}\label{5.21}
\begin{array}{lll}
\hat{N}=&\varepsilon^{\frac{3}{2}}\bigg[-\psi''(z)h^1_2-\psi''(1-z)h^{u,1}_2\\
&+\psi(z)\bigg(\frac{1}{2}\partial_{z}^2 u^0_1(t;0)Z^2\partial_xh^1_2-\partial_x^2h^1_2\bigg)\\
&+\psi(1-z)\bigg(\frac{1}{2}\partial_{z}^2 u^0_1(t;1)(Z^u)^2\partial_xh^{u,1}_2-\partial_x^2h^{u,1}_2\bigg)\\
&-\psi(z)\bigg(\frac{1}{2}\partial_{z}^2 H^0_1(0)Z^2\partial_x\theta^1_2-\partial_x^2\theta^1_2\bigg)\\
&-\psi(1-z)\bigg(\frac{1}{2}\partial_{z}^2 H^0_1(1)(Z^u)^2\partial_x\theta^{u,1}_2-\partial_x^2\theta^{u,1}_2\bigg)\bigg].
\end{array}
\end{equation}

Similar to that in Section \ref{pf}, we introduce the error solutions
$$(\hat{u}^{err},\hat{h}^{err}):=(u^\varepsilon-\tilde{u}^{a,1},H^\varepsilon-\tilde{h}^{a,1}),$$
then the equations for $(\hat{u}^{err},\hat{h}^{err})$ read  as
\begin{equation}\label{5.22}
\left \{
\begin{array}{lll}
\partial_t \hat{u}^{err}_1-\varepsilon \partial_{zz} \hat{u}^{err}_1=-(A+B),\\
\partial_t \hat{u}^{err}_2-\varepsilon\Delta_{x,z}\hat{u}^{err}_2+\hat{u}^{err}_1\partial_x \tilde{u}^{a,1}_2+u^\varepsilon_1\partial_x \hat{u}^{err}_2\\
\quad \quad \quad \quad -\hat{h}^{err}_1\partial_x \tilde{h}^{a,1}_2-H^\varepsilon_1\partial_x \hat{h}^{err}_2=-(C+\hat{D}+\hat{E}+\hat{M}),\\
\partial_t \hat{h}^{err}_1-\varepsilon \partial_{zz} \hat{h}^{err}_1=-(F+G),\\
\partial_t \hat{h}^{err}_2-\varepsilon\Delta_{x,z}\hat{h}^{err}_2+\hat{u}^{err}_1\partial_x \tilde{h}^{a,1}_2+u^\varepsilon_1\partial_x \hat{h}^{err}_2\\
\quad \quad \quad \quad -\hat{h}^{err}_1\partial_x \tilde{u}^{a,1}_2-H^\varepsilon_1\partial_x \hat{u}^{err}_2=-(H+\hat{I}+\hat{J}+\hat{N}),\\
(\hat{u}^{err}_1,\hat{u}^{err}_2, \hat{h}^{err}_1, \hat{h}^{err}_2)|_{z=i}=(0,0,0,0), \ \ i=0,1,\\
(\hat{u}^{err}_1,\hat{u}^{err}_2, \hat{h}^{err}_1,\hat{h}^{err}_2)|_{t=0}=(0,0,0,0).
\end{array}
\right.
\end{equation}

By the higher expansion, we can improve the convergence rates of the Theorem \ref{thm:nonslip}.
\begin{theorem}\label{thm3}
Suppose the initial and boundary data, external force satisfy that $u_0, H_0 \in H^m(\Omega)$, $f \in L^\infty(0,T;H^m(\Omega)), \alpha^i,\gamma^i \in H^2(0,T; H^m(\partial{\Omega}))$, $i=0,1, m>8$ and the compatibility conditions (\ref{1.4a0}) and (\ref{1.4b0}). Then there exist positive constants $C>0$, independent of $\varepsilon$, such that for any solution $(u^\varepsilon,H^\varepsilon)$ of (\ref{1.4}) with the initial data $(u_0,H_0)$ and boundary data $\alpha^i,\gamma^i$ in \eqref{1.3a}, it holds that
\begin{equation}\label{5.23}
\Vert (u^\varepsilon-\tilde{u}^{a,1},H^\varepsilon-\tilde{h}^{a,1}) \Vert_{L^\infty(0,T;H^1(\Omega))} \leq C \sqrt{\varepsilon},
\end{equation}
\begin{equation}\label{5.24}
\Vert (u^\varepsilon-\tilde{u}^{a,1},H^\varepsilon-\tilde{h}^{a,1})\Vert_{L^\infty((0,T) \times \Omega))} \leq C \varepsilon^{3/4},
\end{equation}
where $\tilde{u}^{a,1},\tilde{h}^{a,1}$ are defined by (\ref{5.13}).
\end{theorem}
Meawhile, we have
\begin{corollary}\label{cor2}
Under the assumptions of Theorem \ref{thm3}, the following optimal convergence rate holds
\begin{equation}\label{1.77}
\begin{array}{lll}
C \sqrt{\varepsilon} \leq &\Vert u^\varepsilon-u^0-\psi(z)\theta^0-\psi(1-z)\theta^{u,0}\Vert_{L^\infty(0,T;H^1(\Omega))} \leq C \sqrt{\varepsilon},\\
C \sqrt{\varepsilon} \leq &\Vert H^\varepsilon-H^0-\psi(z)h^0-\psi(1-z)h^{u,0}) \Vert_{L^\infty(0,T;H^1(\Omega))} \leq C \sqrt{\varepsilon},
\end{array}
\end{equation}
where $(u^0,H^0)$ is the solution of Problem (\ref{1.5})-(\ref{1.6}) and the constants $C>0$ are independent of $\varepsilon$.
\end{corollary}

The proofs of Theorem \ref{thm3} and Corollary \ref{cor2} are very similar to that in Section \ref{pf}, and we omit them here. In addition, one can follow the similar arguments for the case with perfectly conducting wall condition to obtain the improved convergence rates. Here we omit the details.

\smallskip
{\bf Acknowledgment.}

Ding's research is supported by the National Natural Science Foundation of China (No.11371152, No.11571117, No.11871005 and No.11771155) and Guangdong Provincial Natural Science Foundation (No.2017A030313003). Niu's research is supported by the National Natural Science Foundation of China (No.11471220 and No.11871046) and the key research project of National Natural Science Foundation (No. 11931010). In addition, Niu's research is also supported by key research grant of the Academy for Multidisciplinary Studies, Capital Normal University.

\bigskip

\end{document}